\documentclass[12pt]{amsart}
\usepackage{mathrsfs}

\usepackage[utf8]{inputenc}
\usepackage{amsfonts}
\usepackage{amscd}
\usepackage{amssymb}
\usepackage{amsthm}
\usepackage{amsmath}
\usepackage{blkarray}
\usepackage{color}
\usepackage{comment}
\usepackage{enumitem}
\usepackage{mathrsfs}
\usepackage{stmaryrd}
\usepackage{youngtab}
\usepackage[all]{xy}
\usepackage{mathdots}
\usepackage{multicol}
\usepackage{tikz-cd}
\usepackage{esint}
\usepackage{faktor}
\usepackage{natbib}

\def\cC{\mathcal{C}}

\def\cE{\mathcal{E}}
\def\cF{\mathcal{F}}
\def\cG{\mathcal{G}}

\def\cO{\mathcal{O}}
\def\cP{\mathcal{P}}

\def\cR{\mathcal{R}}

\def\cU{\mathcal{U}}
\def\cV{\mathcal{V}}

\def\cZ{\mathcal{Z}}

\def\CC{\mathbb{C}}

\def\RR{\mathbb{R}}

\def\ZZ{\mathbb{Z}}

\newtheorem*{centralthm}{Main Theorem}

\def\Aut{\mathrm{Aut}}
\def\Iso{\mathrm{Iso}}

\def\End{\mathrm{End}}

\def\id{\mathrm{id}}

\newtheorem{theorem}{Theorem}[section]
\newtheorem{lemma}[theorem]{Lemma}
\newtheorem{corollary}[theorem]{Corollary}
\newtheorem{proposition}[theorem]{Proposition}

\theoremstyle{definition}
\newtheorem{definition}[theorem]{Definition}
\newtheorem{example}[theorem]{Example}

\newtheorem{remark}[theorem]{Remark}

\begin{document}
\bibliographystyle{plain}

\title{The Oka Principle in Higher Twisted K-Theory}
\author{Haripriya Sridharan}
\date{\today}

\maketitle

\begin{abstract}
\noindent The Oka principle is a heuristic in complex geometry which states that, for a wide class of complex-analytic problems concerning Stein spaces, any obstruction to finding a holomorphic solution is purely topological. A classical theorem of H.~Grauert implies that for a reduced Stein space $X$, the natural map $K^{0, \cO}(X) \to K^{0, \cC}(X)$ from ordinary holomorphic K-theory $K^{0, \cO}(X)$ to ordinary topological K-theory $K^{0, \cC}(X)$ is an isomorphism: this is the basic manifestation of the Oka principle in K-theory. 
In this paper, we generalise this theorem to higher twisted K-theory. For a reduced Stein space $X$ and a torsion class $\alpha \in H^3(X,\mathbb{Z})$, we prove that the natural map $K^{-n,\mathcal{O}}_\alpha(X) \to K^{-n,\mathcal{C}}_\alpha(X)$ is an isomorphism for all $n \geq 0$. We introduce the first definition of higher twisted holomorphic K-theory in the literature, defined through a simplicially enriched version of Quillen's $S^{-1}S$ construction. Our parallel construction for topological higher twisted K-theory is a new formulation which is compatible with existing theory. The proof of the main theorem employs Cartan-Grauert cohomological methods and an equivalence, which we prove, between the simplicial symmetric monoidal categories of holomorphic and topological $\alpha$-twisted vector bundles. 
\end{abstract}

\section{Introduction}

The inherent rigidity of holomorphic maps emerges at the most fundamental level of complex analysis -- take, for example, Picard's little theorem, which asserts that the image of a non-constant holomorphic function on the complex plane can omit at most one point. This rigidity makes the holomorphic obstructions to complex-analytic problems appear, a priori, particularly formidable. Remarkably, however, certain problems in complex analysis concerning Stein manifolds (or more general Stein spaces) can be solved by reducing the holomorphic obstructions to purely topological ones, rendering the problem more tractable than one might reasonably expect. 

Here is a fundamental instance of this phenomenon, heuristically referred to as the Oka principle,  in K-theory. In complex geometry, the ordinary K-theory of a reduced Stein space $X$ is the holomorphic K-theory $K^{\cO}(X)$, defined as the Grothendieck group of the abelian monoid of holomorphic vector bundles on $X$. To mark the distinction with classical topological K-theory, the Grothendieck group of topological vector bundles on $X$ is denoted by $K^{\cC}(X)$. The famous Oka-Grauert theorem, which gives a bijection between the isomorphism classes of holomorphic $G$-principal bundles and topological $G$-principal bundles over $X$ for a complex Lie group $G$, implies that the natural map $K^{\cO}(X) \to K^{\cC}(X)$ is an isomorphism of abelian groups. This is the \textit{basic Oka principle in K-theory}.

 The purpose of this work is to elevate this result to the setting of higher, twisted K-theory. In more precise terms, this involves, for an arbitrary torsion class $\alpha \in H^3(X, \ZZ)$: 
 \begin{enumerate}
    \item providing a definition of higher twisted holomorphic K-theory $K_\alpha^{-n, \cO}(X), n \geq 0,$ which generalises $K^\cO(X) = K^{0, \cO}(X)$, the first such definition in the complex geometry literature,
    \item presenting an analogous definition of higher twisted topological K-theory $K_\alpha^{-n, \cC}(X),$ $n \geq 0,$ that is consistent with existing versions of the theory, and
    \item proving the Oka principle in higher, $\alpha$-twisted K-theory; that is, the inclusion $\cO \hookrightarrow \cC$ induces an isomorphism of groups,\[K_\alpha^{-n, \cO}(X) \to K_\alpha^{-n, \cC}(X),\] for all $n \geq 0$. 
 \end{enumerate}
 
 The author wishes to sincerely thank Finnur L\'arusson and Christian Haesemeyer for their invaluable guidance, discussions and insights throughout the development of this work. 

\subsection{Historical context}This interdisciplinary paper lies at the confluence of Oka theory and twisted K-theory, two vibrant areas of mathematical research with rich, classical roots, and long histories shaped by the major currents of modern mathematics -- algebra, topology, analysis and geometry. In this section, we present a brief historical account of the ideas central to our work.

%THIS The inherent rigidity of holomorphic maps emerges at the most fundamental level of complex analysis -- take, for example, Picard's little theorem, which asserts that the image of a non-constant holomorphic function on the complex plane can omit at most one point. This rigidity makes the holomorphic obstructions to complex-analytic problems appear, a priori, particularly formidable. Remarkably, however, certain problems in complex analysis concerning Stein manifolds (or more general Stein spaces) can be solved by reducing the holomorphic obstructions to purely topological ones, rendering the problem more tractable than one might reasonably expect. 
The field of Oka theory originated in a pioneering paper of K.~Oka of 1939 \cite{10.32917/hmj/1558749869}, where he showed that a holomorphic line bundle on a domain of holomorphy is holomorphically trivial if and only if it is topologically trivial. Oka's surprising insight was followed in 1958 by H.~Grauert's significant generalisation \cite{grauert1958} to principal bundles over reduced Stein spaces, with arbitary complex Lie groups for fibres. In his famous and elegant exposition \cite{cartan1958espaces} of Grauert's work, H.~Cartan provided a slight strengthening of the theorem so as to apply to generalised principal bundles over reduced Stein spaces.  These important ideas evolved into a general heuristic principle known as the \textit{Oka principle}: that on a Stein space, complex-analytic problems which are cohomologically formulated have only topological obstructions \cite{forstneric2010survey}. 
 
Three decades later, modern Oka theory was born with Gromov's seminal 1989 paper \cite{gromov1989oka}, which marked a shift away from cohomological methods towards more geometric, homotopy-theoretic methods. These methods were better suited to the discovery of more general Oka principles, and capable of handling the challenging new questions that were emerging in Stein geometry -- such as Forster's conjecture of 1970 \cite{forster1970plongements}, solved for ${n \geq 2}$ by Eliashberg and Gromov in 1992 \cite{eliashberg1992embeddings}, and later improved for Stein manifolds of odd dimension by J.~Sch\"{u}rmann \cite{schurmann1997embeddings}. Since the year 2000, Oka theory has developed significantly through the ongoing efforts of Forstneri\v{c} and others, and forms an active subfield of complex geometry in its own right; see the book \cite{forstnerivc2011stein} for a unified and up-to-date account of the modern theory. 

Elsewhere in mathematics, around the same period that Grauert introduced his Oka principle, K-theory found its origins in Grothendieck's work on his generalisation of the Riemann-Roch theorem \cite{borel1958theoreme}. For an algebraic variety $X$, Grothendieck introduced the abelian group $K(X)$ -- in modern notation, $K^0(X)$, the \textit{Grothendieck group} of $X$ -- generated by isomorphism classes $[E]$ of locally free sheaves on $X$, subject to relations $[E] = [E'] + [E'']$ for each short exact sequence $0 \to E' \to E \to E'' \to 0$. Tensor product of sheaves endows this group with the structure of a commutative ring. When $X$ is an affine variety with coordinate ring $R$, the Grothendieck group is denoted by $K(R)$; or, in modern notation, $K_0(R)$. In this case, since locally free sheaves are equivalent to finitely generated projective $R$-modules, and every short exact sequence splits, the defining relation is simply $[E' \oplus E''] = [E'] + [E'']$.
Building on Grothendieck's work, M.~Atiyah and F.~Hirzebruch \cite{atiyah1961vector} translated the construction to topology, defining $K^0(X)$, for a suitable topological space $X$, using topological vector bundles in place of locally free sheaves. Using Bott periodicity and suspensions, Atiyah and Hirzebruch extended it to a graded ring $K^*(X)$, known as higher topological K-theory, which became one of the first and most influential examples of a generalised cohomology theory  -- that is, one which satisfies all the Eilenberg-Steenrod axioms except the dimension axiom. 

The development of higher topological K-theory raised the question: could there be analogous higher algebraic K-groups $K_n(R)$, and how should they be defined? Perhaps unsurprisingly, progress on this question was substantially slower, owing to the rigidity of the algebraic setting. The first accepted definition of $K_1(R)$ was provided by H.~Bass in \cite{bass1964k}, which was followed by J.~Milnor introducing $K_2(R)$ in \cite{milnor1971introduction}. 

The breakthrough finally came in the early 1970s, with D.~Quillen's introduction of multiple, non-trivially equivalent definitions of higher algebraic K-theory. The first was the $+$-construction, announced in \cite{quillen1970cohomology}. For a ring $R$, the $+$-construction gives abelian groups $K_n(R)$ for all $n \geq 1$, matching Bass' definition for $n=1$ and Milnor's definition for $n=2$.
This was shortly followed by a second construction due to Quillen -- reported later by D.~Grayson in \cite{grayson2006higher} -- known as the $S^{-1}S$ construction, or the localisation construction, for the higher algebraic K-theory of a symmetric monoidal category $S$. The definitions of twisted K-theory we use in this paper are derived from this construction, which carries various appealing features. Firstly, the $S^{-1}S$ construction naturally generalises the way one introduces formal inverses to complete an abelian monoid into a group; in other words, it lifts Grothendieck's group completion construction to symmetric monoidal categories. Secondly, it does not require $K_0$ to be considered separately from the higher groups, as is necessary with the $+$-construction. It is also an example of an infinite loop space machine -- a construction which produces a group completion of a homotopy-associative H-space and gives it the structure of an infinite loop space. It is well known that infinite loop space machines accept topological inputs \cite[pp.~331, 338]{weibel2013k}; this is an important fact, as we enrich the relevant symmetric monoidal categories in topological spaces to prove the main theorem.

The most influential and general of Quillen's definitions of higher K-theory is the $Q$-construction, published in 1972 \cite{quillen2006higher}, allowing for categories in which exact sequences do not need to split. We do not work directly with this definition, as exact sequences do split in the categories we are concerned with. It is, however, an amazing fact that all of these definitions of K-theory agree \cite{grayson2006higher, quillen2006higher}. The early history of K-theory is rich, and we have only touched on the parts which relate most closely with the mathematics in this paper; for a more detailed history see \cite{weibel1999development} or \cite{grayson2013quillen}. 

Twisted K-theory, which originated in the work of P.~Donovan and M.~Karoubi \cite{donovan1970graded}, and J.~Rosenberg \cite{rosenberg1989continuous}, is a sophisticated extension of classical K-theory that has elicited strong research interest in both mathematics and mathematical physics over the past two decades, particularly following Atiyah and Segal's foundational geometric reformulations in \cite{atiyah2004twisted} and \cite{atiyah2006twisted}. To each cohomology class $\alpha \in H^3(X, \ZZ)$ of a topological space $X$ is associated an $\alpha$-twisted K-group $K^0_\alpha(X)$. In keeping with the richly interdisciplinary character of K-theory, the group $K^0_\alpha(X)$ may be motivated in various ways, and admits multiple definitions which are equivalent in a deep, non-trivial manner. These include extensively studied formulations in terms of bundle gerbes \cite{bouwknegt2002twisted}, $C^*$-algebras, and Fredholm operators (cf. \cite{atiyah2004twisted} and \cite{atiyah2006twisted}). In our work, we adopt -- and generalise to higher K-theory -- the relatively direct approach based on $\alpha$-twisted vector bundles, introduced by Karoubi in \cite{karoubi2012twisted}, where the twisting cohomology class $\alpha \in H^3(X, \ZZ)$ is torsion.  An $\alpha$-twisted vector bundle is, to borrow A.~C\u{a}ld\u{a}raru's phrasing \cite{caldararu2000derived}, a collection of vector bundles and gluing functions `with the apparent defect that they don't quite match up'; that is, the cocycle condition includes the factor $\alpha$. (The precise definition we adopt is necessarily a little more intricate than Karoubi's definition -- cf. Definition \ref{def: cat twisted} and Remark \ref{rem: comp with Karoubi}.) Analogously to classical topological K-theory, $K^0_\alpha(X)$ is defined as the Grothendieck group of isomorphism classes of $\alpha$-twisted vector bundles on $X$. 

The aspects of twisted K-theory one focuses on tend to reflect the broader mathematical context and intended applications of the work -- while non-torsion twisting classes arise naturally in string theory \cite{karoubi2008twisted, atiyah2004twisted}, our focus on torsion twisting classes aligns more closely with the study of twisted sheaves, Azumaya algebras and the Brauer group in algebraic geometry. 
The following basic question can be illuminating: what is the obstruction to lifting a holomorphic $PGL_n\CC$-principal bundle, over a complex analytic space $X$, to a $GL_n\CC$-principal bundle? A holomorphic $PGL_n\CC$-principal bundle is, up to isomorphism, equivalent to an element $[E] \in H^1(X, PGL_n\cO)$. The long exact cohomology sequence $H^1(X, \cO^*) \to H^1(X, GL_n\cO) \to H^1(X, PGL_n\cO) \xrightarrow{\partial} H^2(X, \cO^*)$ reveals that for $[E] \in H^1(X, PGL_n\cO)$, the obstruction to finding a lift $[F] \in H^1(X, GL_n\cO)$ is precisely the $n$-torsion element $\alpha = \partial([E])$. This observation leads to the study of the cohomological Brauer group $\mathrm{Br}'(X) = H^2(X, \cO^*)$, which for Stein spaces is isomorphic to $H^3(X, \ZZ)$. It is a well-known result that the Brauer group $\mathrm{Br}(X)$ of isomorphism classes of Azumaya algebras under tensor product injects into $\mathrm{Br}'(X)$, and the image corresponds precisely with the obstructions $\partial([Y])$ arising in this context, for all $n \geq 0$ \cite[Thm. 1.18]{caldararu2000derived}. While our focus lies with torsion $\alpha$, we note that Karoubi's definition includes the non-torsion case by considering, instead of vector bundles, bundles with infinite-dimensional Hilbert spaces for fibres \cite{karoubi2012twisted}.
\subsection{Structure and outlook} Outside of this introduction, the paper is organised into four sections. 
 Section 2 provides some background results, including, for the reader's benefit, a complete proof of the category-theoretic equivalence between ordinary and generalised principal bundles, previously only sketched by M.~Murray in the note \cite{murray2010bundle}. We also prove a parametric version of the Cartan-Oka-Grauert principle. While the result is now understood as a special case of a more general Oka principle within modern, Gromov-style Oka theory, we derive it directly here using classical methods of Cartan and Grauert.

Section 3 develops the homotopy-theoretic machinery needed for our main results. We study a simplicial enrichment of Quillen's $S^{-1}S$ construction and prove that it preserves weak equivalences between simplicial symmetric monoidal categories under certain mild conditions (analogous to those in Quillen's discrete construction \cite{grayson2006higher}). This yields the key consequence that weak equivalences between such categories induce isomorphisms of K-groups, which we apply in Section 4 to establish the Oka principle.

Section 4, which forms the core of the paper, establishes the Oka principle for higher twisted K-theory. We construct simplicially enriched, symmetric monoidal categories of holomorphic and topological $\alpha$-twisted vector bundles over Stein spaces. The definitions we provide are necessarily more intricate than Karoubi's definition in \cite{karoubi2012twisted}, as holomorphic constructions do not have available to them the flexibility of topological ones. Using the simplicially enriched $S^{-1}S$ contruction we introduce two new definitions: higher holomorphic $\alpha$-twisted K-theory for a torsion twisting class $\alpha$ (the first such definition in the literature beyond the $K^0$ case, defined in \cite{mathai2003chern} using different methods), and a new formulation of higher, topological $\alpha$-twisted K-theory that can be shown consistent with existing versions. We then prove the main theorem: that the inclusion of the category of holomorphic $\alpha$-twisted vector bundles into the category of topological $\alpha$-twisted vector bundles induces an isomorphism between their respective higher $\alpha$-twisted $K$-groups. %The proof involves showing the statement to be a special case of the main theorem of Chapter 4; which, in turn, is achieved by translating the methods of Cartan and Grauert to this more nuanced context, and using the parametric version of the Oka principle for generalised principal bundles, proven in Chapter 3. 

In a final short section, Section 5, we outline how our definition of higher, $\alpha$-twisted topological K-theory coincides with existing, well-known definitions in the literature. As such, the Oka principle that is the subject of this work -- beyond providing yet another remarkable instance of the relationship between complex geometry and topology -- fits in a deep sense within the rich, broader setting of twisted K-theory.

This work opens several promising avenues for future research. One natural direction would be to explore whether the Oka principle extends to non-torsion twisting classes in $H^3(X,\ZZ)$ -- a question that may be of particular interest for mathematical physics applications; in which case, the definition of twisted K-theory would need to be based on infinite-dimensional bundles, with Hilbert spaces for fibres. 
Another natural research direction which would both complement and give new perspectives to our work, is examining how the Oka principle manifests through other K-theory machines, such as Quillen's $Q$-construction, or Waldhausen's $wS_{\bullet}$-construction. 

The connection to algebraic K-theory deserves particular exploration: to what extent, and in which special cases, does the Oka principle hold when the categories are not topologically enriched? This is a subtle question -- Ivarsson and Kutzschebauch showed in \cite{ivarsson2012holomorphic} that an untwisted version of the Oka principle fails under certain situations even at $K_1$, when working with discrete categories. We crucially depend on enrichment in our work -- the key results of Section 3 involve bisimplicial sets that arise from simplicial enrichment. Understanding more precisely how this enables our results, and the degree to which Oka principle does or does not hold in algebraic twisted K-theory, is a topic with deep potential.

\section{Background}

\subsection{Higher K-theory and Quillen's $S^{-1}S$ construction} \label{sub: s-1s}

In this paper, we employ an enriched version of Quillen's $S^{-1}S$ construction to define and study holomorphic and topological higher twisted K-theory. The $S^{-1}S$ definition for higher K-theory provides what is perhaps the most natural way to extend the Oka-Grauert principle for generalised principal bundles to higher twisted K-theory.

We recall here the classical version of the $S^{-1}S$ construction. Let $(S,\square)$ be a small symmetric monoidal category whose morphisms are all isomorphisms. The $S^{-1}S$ construction produces a new symmetric monoidal category in which objects of $S$ become invertible, up to homotopy, under $\square$. Objects in $S^{-1}S$ are pairs $(x,y)$ of objects in $S$, interpreted as formal differences $x-y$. A morphism $(x,y) \to (z,w)$ is an equivalence class $[s,f,g]$, where $s \in \text{ob}(S)$, $f: s \square x \to z$, and $g: s \square y \to w$ are morphisms in $S$. Two triples $(s,f,g)$ and $(s',f',g')$ are equivalent when there exists an isomorphism $\alpha: s \to s'$ such that $f = f' \circ (\alpha \square \text{id}_x)$ and $g = g' \circ (\alpha \square \text{id}_y)$.

\begin{theorem}[Quillen \cite{grayson2006higher}]
When translations in $S$ are faithful, $|NS^{-1}S|$ is the group completion of $|NS|$ as an H-space, where $N$ is the nerve functor and $|\cdot|$ denotes geometric realisation.
\end{theorem}

Based on this construction, the higher K-theory groups of $S$ are defined as:

$$K_n(S) = \pi_n(|NS^{-1}S|), \quad n \geq 0.$$

These groups coincide with the K-groups of Quillen's $+$-construction when $S$ is the category of finitely generated projective modules over a ring. However, in the context of topological or holomorphic higher twisted K-theory, the discrete setting of $S^{-1}S$ construction above is not sufficient, as the morphism sets carry essential topological or complex-analytic structure. To address this, we will need to enrich the construction in the category of simplicial sets (see the following section); or equivalently, since we are concerned primarily with weak equivalences, in topological spaces. 
\subsection{Generalised principal bundles}
%Central to the main pursuit of this work is the notion of a \textit{generalised principal bundle}, formally introduced by Cartan \cite{cartan1958espaces} in his exposition of Grauert's Oka principle for principal bundles.

The proof of the main theorem in this work -- the Oka principle in higher twisted K-theory -- relies centrally on the notion of a \textit{generalised principal bundle}, formally introduced by Cartan \cite{cartan1958espaces} in his exposition of the Oka-Grauert principle for principal bundles on a Stein space. Its essential distinction from an ordinary principal bundle lies in the fact that a generalised principal bundle is acted on by structure group bundle, rather than by a fixed structure group; locally, however, it remains isomorphic to an ordinary principal bundle. Here is a more precise definition:

\begin{definition}
  Let $X$ be a Hausdorff, paracompact space, and $p: E \to X$ a topological group bundle with base $X$. An \textit{$E$-principal bundle} is a fibre bundle $\pi: P \to X$ together with a continuous map $\rho: P \times_X E \to P$, such that:
\begin{enumerate}
\item each induced map $\rho_x: P_x \times E_x \to P_x$ on fibres is a right action of $E_x$ on $P_x$, and
\item for each $x \in X$, there exists a neighborhood $U \subset X$ over which $E$ is trivialisable, and $P|_U$ is isomorphic to a principal $G$-bundle with respect to the action of $G$ induced by a trivialisation $E|_U \simeq U \times G$.
\end{enumerate}

It can be shown that for a topological group $G$ and a Hausdorff, paracompact space $X$, the category of generalised principal bundles over $X$ with varying $G$-group structure bundles is equivalent to the category of ordinary principal bundles over $X$ with structure group $G \rtimes \Aut(G)$ (Theorem \ref{Murray correspondence}). This fundamental correspondence, noted by M.~Murray \cite{murray2010bundle}, allows us to gain several insights into generalised principal bundles through classical principal bundle theory. For the reader's benefit, we provide here an account of this result, which, though known to experts, does not appear to have a complete presentation in the literature. 

We first detail the natural correspondence between $G$-group bundles and $\Aut(G)$-principal bundles. 
While the structures we work with in this section are assumed to be topological, we note with importance that everything presented here can be adapted to the smooth or holomorphic setting with weak sufficient conditions placed upon $G$. This is to ensure that if $G$ is a smooth or complex Lie group, then $\Aut(G)$ is a smooth or complex Lie group respectively. One such example of a sufficient condition is that $\pi_0(G)$ is finitely generated (see \cite[Section 15.4]{hilgert2012structure}). 

\begin{proposition}\label{prop:pre} The category of $G$-group bundles over a base $X$ is equivalent to the category of $\Aut(G)$-principal bundles over $X$. 
\end{proposition}

In short, the equivalence is forged by sending a $G$-group bundle to its frame bundle, and an $\Aut(G)$-principal bundle to the associated bundle for the natural left action of $\Aut(G)$ on $G$. Expressly, this means the following. 

Let $\cG$ be a $G$-group bundle over $X$. Recall that the \textit{frame bundle} $F(\cG)$ of $\cG$, given by the disjoint union \[F(\cG) = \coprod_{x \in X} \mathrm{Iso}(G, \cG_x),\] where $\mathrm{Iso}(G, \cG_x)$ is the set of topological group isomorphisms from $G$ to the fibre $\cG_x$ of $\cG$ over $x \in X$. The topology and bundle structure on $F(\cG)$ is determined by the group bundle $\cG$ as follows. Let $h_U: \cG_U \to U \times G$ be a trivialisation; if $x \in U$, it induces an isomorphism $h_x: \cG_x \to G$, and we have a bijection
\[
\mathrm{Iso}(G, \cG)_U = \coprod_{x \in U} \mathrm{Iso}(G, \cG_x) \rightarrow U \times \Aut(G), \quad (x, \phi_x) \mapsto (x, h_x \circ \phi_x),
\] through which $\mathrm{Iso}(G, \cG)_U$ acquires a topological structure. The topology on $F(\cG)$ is the final topology coinduced by the inclusion maps $\mathrm{Iso}(G, \cG)_U \hookrightarrow F(\cG)$. We also have a natural projection map $F(\cG) \to X, (x, \phi_x) \mapsto x$. Thus, $F(\cG)$ is a principal bundle over $X$ with structure group $\Aut(G)$.

Let us denote by $\cE_GX$ the  category of $G$-group bundles over $X$, and by $\cP_{\Aut(G)}X$ the category of $\Aut(G)$-principal bundles over $X$. We then have a functor
\[
\Phi: \cE_GX \to \cP_{\Aut(G)}X,
\] which sends a $G$-group bundle $\cG$ to $F(\cG)$, and sends a morphism $f: \cG \to \tilde{\cG}$ to $\Phi(f): F(\cG) \to F(\tilde{\cG}), (x, \phi_x) \mapsto (x, f \circ \phi_x)$. We also have a functor in the other direction,
\[
\Psi: \cP_{\Aut(G)}X \to \cE_GX,
\]
which takes an $\Aut(G)$-principal bundle $P$ to the associated bundle\footnote{~Say we have a $G$-principal bundle $\pi: P \to X$ and a left action of $G$ on some $F$. The \textit{associated bundle} to $P$, with fibre $F$, is given by $P \times_G F :=  (P \times F)/{\sim}$, where the equivalence relation $\sim$ is defined by $(pg, f) \sim (p, gf)$ for all $(p, f) \in P \times F$ and $g \in G$. The canonical projection is  $(p, f) \mapsto \pi(p)$, and each of the fibres carries the underlying structure of $F$.} $P \times_{\Aut(G)} G$ for the left action of $\Aut(G)$ on $G$, and maps an $\Aut(G)$-equivariant morphism $f: P \to \tilde{P}$ to $\Psi(f): P \times_{\Aut(G)}G \to \tilde{P} \times_{\Aut(G)}G, [p, g] \mapsto [f(p), g]$ (this is well defined because of the $\Aut(G)$-equivariance of $f$).

Observe that $F(\cG) \times_{\Aut(G)} G$ is isomorphic to $\cG$ in the category $\cE_GX$, through the fibre-preserving maps $F(\cG) \times_{\Aut(G)} G \to \cG, \left[(x, \phi_x), g\right] \mapsto \phi_x(g)$, and $\cG \to F(\cG) \times_{\Aut(G)} G, k \mapsto \left[(x, h_x^{-1}), h_x(k)\right],$ where $k \in \cG_x$ and $h_x$ is any isomorphism $\cG_x \to G$. Conversely, there is a natural $\Aut(G)$-equivariant morphism  $P \to \cF(P \times_{\Aut(G)}G)$ in the category $\cP_{\Aut(G)}X$, given by sending $p \in P_x$ to $(x,\varphi_x) \in \mathrm{Iso}(G, P_x \times_{\Aut(G)}G)$, where \[\varphi_x: G \to P_x \times_{\Aut(G)}G, \quad g \mapsto \left[p, g\right]. \] (It is easy to see that $\varphi_x$ is a well-defined group isomorphism.) As a morphism of $\Aut(G)$-principal bundles, it is in fact, by classical theory, an isomorphism of $\Aut(G)$-principal bundles. Naturality is a slightly cumbersome but straightforward verification, which we omit presenting here for conciseness.

These arguments establish Proposition \ref{prop:pre} -- $\cE_GX$ and $\cP_{\Aut(G)}X$ are equivalent categories. With this, we may proceed to state and prove the Murray correspondence between generalised and ordinary principal bundles.

Let $\cR_GX$ denote the category of generalised principal bundles over $X$, with varying structure group bundles that are $G$-group bundles. An object of $\cR_GX$ is a pair $(P, \cG)$ where $\cG$ is a $G$-group bundle and $P$ is a $\cG$-principal bundle. A morphism $(f, \chi)$ from $(P_1, \cG_1)$ to $(P_2, \cG_2)$ in this category consists of continuous\footnote{\, We are, as we have assumed from the outset, working in the topological setting -- analogous definitions hold in the smooth and holomorphic settings.} fibre-preserving maps $f: P_1 \to P_2$ and $\chi: \cG_1 \to \cG_2$ such that, if $x \in X, \, p \in P_x$ and $g \in (\cG_1)_x$, then $f(pg) = f(p)\chi(g)$.

\begin{theorem} \label{Murray correspondence}
The category $\cR_GX$ is equivalent to $\cP_{G \rtimes \Aut(G)}X$, the category of $(G \rtimes \Aut(G))$-principal bundles over $X$.
\end{theorem}
The operation on $G \rtimes \Aut(G)$ is the canonical semi-direct product, $(g_1, \phi_1)(g_2, \phi_2) = (g_1\phi_1(g_2), \phi_1 \circ \phi_2)$. Before giving a careful proof of Theorem \ref{Murray correspondence}, we make a few remarks on the natural reasoning underlying the correspondence. First, consider an object $(P, \cG)$ of $\cR_GX$. We know that $\cG$ determines an $\Aut(G)$-principal bundle, namely the frame bundle $F(\cG)$; so, the fibre bundle $P \times_X F(\cG)$ carries the information of $(P, \cG)$ but with one significant gap -- we cannot reconstruct from it the action of $\cG$ on $P$. However, we can, informally speaking, restore this information (the $\cG$-action on $P$) to the fibre bundle $P \times_X F(\cG)$ by upgrading it to a $(G \rtimes \Aut(G))$-principal bundle -- specifically, we define a right action of $G \rtimes \Aut(G)$ on $P \times_X F(\cG)$ by $(p, \phi)(g, \xi) = (p\phi(g), \phi \circ \xi )$, where $(p, \phi) \in P \times_X F(\cG)$ and $(g, \xi) \in G \rtimes \Aut(G)$. This makes $P \times_X F(\cG)$ into a $(G \rtimes \Aut(G))$-principal bundle, an `unrolling' of $P$ into an ordinary principal bundle. Second, consider a $(G \rtimes \Aut(G))$-principal bundle $R$ -- an arbitrary object of $\cP_{G \rtimes \Aut(G)}X$. Such an object encompasses precisely the data needed to construct a generalised principal bundle $P$ with structure group bundle $\cG$, a $G$-group bundle; one way to intuitively see this is as follows. We know that the fibres of $R$ are $(G \rtimes \Aut(G))$-torsors; a natural candidate for $P$, then, is the induced fibre bundle ${R}/{\Aut(G)}$ (that is, the space $R$ modulo the equivalence relation $\sim$, where $x \sim y$ if and only if there exists $(e, \chi) \in G \rtimes \Aut(G)$ such that $x(e, \chi) = y$). The reason that this is the first and only obvious choice for $P$ is that each of its fibres ${R_x}/{\Aut(G)}$, where $x \in X$, is a $G$-torsor, since the space ${(G \rtimes \Aut(G))}/{\Aut(G)}$ of right $\Aut(G)$-cosets of $G \rtimes \Aut(G)$ is a $G$-torsor. Indeed, we have a well-defined action,
\[G \times \left({(G \rtimes \Aut(G))}/{\Aut(G)}\right) \to {(G \rtimes \Aut(G))}/{\Aut(G)}, \quad (g, [h, \varphi]) \mapsto [(g, \id) \cdot (h, \varphi)],\] which is easily verified to be free and transitive. 

In general, there is no well-defined (let alone fibrewise free and transitive) action of $G$ on the fibre bundle ${R}/{\Aut(G)}$; certainly, the action of $G$ on $R$ does not pass to ${R}/{\Aut(G)}$, for the action of $G$ on any fibre of $R$ (that is, of $G$ on $R_x$) does interact with $(e , \varphi)$ if $r(e, \varphi) \in R_x$. However, $R$ similarly induces another fibre bundle ${R}/{G}$, which is an $\Aut(G)$-principal bundle (each fibre is homeomorphic to ${(G \rtimes \Aut(G))}/{\Aut(G)} \simeq \Aut(G)$). When ${R}/{G}$ is repackaged into the $G$-group bundle $\left({R}/{G}\right) \times_{\Aut(G)}G \simeq R \times_{G \rtimes \Aut(G)}G$ as per Proposition \ref{prop:pre}, it acts on $P$ through the well-defined, fibrewise free and transitive action given by
\[ [r]_{\Aut(G)}[r, g]_{G \rtimes \Aut(G)} = [r(g, \id_G)]_{\Aut(G)},
\]
where $r \in R$ and $g \in G$. Thus we set $\cG = R \times_{G \rtimes \Aut(G)}G$. 
(Note that we have included the subscripts $\Aut(G)$ and $G \rtimes \Aut(G)$ for clarity here, but if the context makes the equivalence class clear, we will omit the subscripts.)

We now prove Theorem \ref{Murray correspondence} in detail. To an object $(P, \cG)$ of $\cR_GX$ we associate the object $P \times_X F(\cG)$ of $\cP_{G \rtimes \Aut(G)}X$. The group action of $G \rtimes \Aut(G)$ on $P \times_X F(\cG)$ is as above; that this action is free and transitive follows easily from the fact that the action of $\cG$ on $P$ is fibrewise free and transitive. Let $(f, \chi)$ be a morphism in the category $\cR_GX$ from $(P_1, \cG_1)$ to $(P_2, \cG_2)$. As we know, $\chi: \cG_1 \to \cG_2$ determines a morphism $\Phi(\chi): F(\cG_1) \to F(\cG_2)$ between the associated frame bundles. We then associate to $(f, \chi)$ the morphism $(f, \Phi(\chi)): P_1 \times_X F(\cG_1) \to P_2 \times_X F(\cG_2)$ of $(G \rtimes \Aut(G))$-principal bundles -- expressly given by $(p, \psi) \mapsto (f(p), \chi \circ \psi)$. It is easy to check that this morphism is $(G \rtimes \Aut(G))$-equivariant: {\begin{align*}(p, \psi)(g, \xi) &= (p\psi(g), \psi \circ \xi) \\ &\mapsto (f(p)\,\chi \circ \psi (g), \chi \circ \psi \circ \xi) \\ &= (f(p), \chi \circ \psi)(g, \xi). \end{align*}} This defines a functor from $\Gamma$ from $\cR_GX$ to $\cP_{G \rtimes \Aut(G)}X$.

Consider, conversely, an object $R \to X$ of $\cP_{G \rtimes \Aut(G)}X$. The natural homomorphism $G \rtimes \Aut(G) \to \Aut(G)$ induces the bundle $R \times_{G \rtimes \Aut(G)}\Aut(G)$, where the left action of $G \rtimes \Aut(G)$ on $\Aut(G)$ is given by $(g, \alpha)\varphi = \alpha \circ \varphi$. This is an $\Aut(G)$-principal bundle through the free and transitive right action of $\Aut(G)$ defined by $[r, \varphi]\psi = [r, \varphi \circ \psi]$; by Proposition \ref{prop:pre}, it in turn induces the $G$-group bundle $R \times_{G \rtimes \Aut(G)}G$, where the left action of $G \rtimes \Aut(G)$ on $G$ is given by $(g, \varphi)h = \varphi(h)$. (We may notice that the $G$-group bundle we get from Proposition \ref{prop:pre} is $(R \times_{G \rtimes \Aut(G)}\Aut(G)) \times_{\Aut(G)}G$, but this rather cumbersome bundle is naturally isomorphic to $R \times_{G \rtimes \Aut(G)}G$ -- we have a well-defined isomorphism given by $(R \times_{G \rtimes \Aut(G)}\Aut(G)) \times_{\Aut(G)}G \to R \times_{G \rtimes \Aut(G)}G, [r, \psi, h] \mapsto [r, \psi(h)]$, with inverse $R \times_{G \rtimes \Aut(G)}G \to (R \times_{G \rtimes \Aut(G)}\Aut(G)) \times_{\Aut(G)}G, [r, h] \mapsto [r, \id_G, h]$.) Let us denote this $G$-group bundle by $\cG$. Now, the $(G \rtimes \Aut(G))$-principal bundle $R$ induces the $\cG$-principal bundle $P := \faktor{R}{\Aut(G)}$. We define a right action of $\cG$ on $P$ by
\[ [r]_{\Aut(G)}[r, g]_{G \rtimes \Aut(G)} = [r(g, \id_G)]_{\Aut(G)}.
\] Firstly, this action is well defined: given $x \in X$ and $r \in R_x$, any element of the fibre $\cG_x$ has a representative of the form $(r, g)$, for a unique $g \in G$. Secondly, to see that the action is transitive, let $[r]_{\Aut(G)}, [r']_{\Aut(G)} \in P_x$ for some $x \in X$ and $r, r' \in R_x$; since the action of $G \rtimes \Aut(G)$ on $R_x$ is transitive, there exists $(g, \varphi) \in G \rtimes \Aut(G)$ such that $r(g, \varphi) = r'$ in $R_x$. This implies that {\begin{align*}[r']_{\Aut(G)} = [r(g, \varphi)]_{\Aut(G)} = [r(g, \id_G)]_{\Aut(G)} = [r]_{\Aut(G)}[r, g]_{G \rtimes \Aut(G)}.\end{align*}} Finally, observe that the action is free: if $[r(g, \id_G)]_{\Aut(G)} = [r]_{\Aut(G)}$, then there exists some $(g, \varphi) \in G \rtimes \Aut(G)$ such that $r(g, \varphi) = r$ in $R_x$, which in turn implies that $(g, \varphi) = (e, \id_G)$ -- in particular, $g = e$. This shows that $(P, \cG)$ is indeed an object belonging to the category $\cR_GX$.

We now turn our attention to morphisms of $P_{G \rtimes \Aut(G)}X$. Let $\lambda: R_1 \to R_2$ be a morphism of $(G \rtimes \Aut(G))$-principal bundles. Per the above construction, $R_1$ and $R_2$ give rise to objects $(P_1, \cG_1)$ and $(P_2, \cG_2)$, respectively, of $\cR_GX$. We know that $\lambda$ is $(G \rtimes \Aut(G))$-equivariant, and when restricted to any fibre, is an isomorphism of $G \rtimes \Aut(G)$-torsors -- in particular, an injection. As a result, we obtain well-defined fibre-preserving maps $f: P_1 \to P_2, [r]_{\Aut(G)} \mapsto [\lambda(r)]_{\Aut(G)}$ and $\chi: \cG_1 \to \cG_2, \, [r, g]_{G \rtimes \Aut(G)} \mapsto [\lambda(r), g]_{G \rtimes \Aut(G)}$. We verify easily that $(f, \chi)$ is a genuine morphism of $\cR_GX$: {\begin{align*} f([r]_{\Aut(G)}[r,h]_{G \rtimes \Aut(G)}) &= f([r(h, \id_G)]_{\Aut(G)}) \\ &= [\lambda(r(h, \id_G))]_{\Aut(G)} \\ &= [\lambda(r)(h, \id_G)]_{\Aut(G)} \\ &= [\lambda(r)]_{\Aut(G)}[\lambda(r),h]_{G \rtimes \Aut(G)} \\ &= f([r]_{\Aut(G)})\chi([r, h]_{G \rtimes \Aut(G)}). \end{align*}}
We now have a functor from $\cP_{G \rtimes \Aut(G)}X$ to $\cR_GX$; we denote it by $\Lambda$. 

In order to see that $\cR_GX$ and $\cP_{G \rtimes \Aut(G)}X$ are equivalent categories, it remains to show that {\begin{enumerate} 
\item[(1)] for any $(P, \cG)$ in $\cR_GX$, $\Lambda \circ \Gamma (P, \cG)$ is naturally isomorphic to $(P, \cG)$, and 
\item[(2)] for any $R \in \cP_{G \rtimes \Aut(G)}X$, $\Gamma \circ \Lambda (R)$ is naturally isomorphic to $R$. \end{enumerate}}
\noindent Take an object $(P, \cG)$ of $\cR_GX$; then, $$\Lambda \circ \Gamma (P, \cG) = \left({\left(P \times_X F(\cG) \right)}/{\Aut(G)},  \left(P \times_X F(\cG) \right) \times_{G \rtimes \Aut(G)}G \right).$$ For (1), we show that there is an isomorphism $\xi: \cG \to \left(P \times_X F(\cG) \right) \times_{G \rtimes \Aut(G)}G$ of $G$-group bundles and a morphism $f: P \to {\left(P \times_X F(\cG) \right)}/{\Aut(G)}$ that is $\xi$-equivariant (in the sense that if $x \in X$, then $f(pg) = f(p)\xi(g)$ for all $p \in P_x$ and $g \in \cG_x$)\footnote{~Analogously to the case of ordinary principal bundles, if we have a morphism of generalised principal bundles over a common base, together with an isomorphism between their structure group bundles, the generalised principal bundles are in fact isomorphic. The proof is also entirely analogous. Take $(P_1, \cG_1)$ and $(P_2, \cG_2)$ over the base $X$; say $\phi: \cG_1 \to \cG_2$ is an isomorphism and $f: P_1 \to P_2$ is a $\phi$-equivariant morphism. When restricted to a fibre over $x \in X$, $f_x$ is a $\phi_x$-equivariant isomorphism of torsors. The inverse $f^{-1}: P_2 \to P_1$ of $f$ is then defined by $p \mapsto f_x^{-1}(p)$, given any $p$ in the fibre of $P_2$ over $x$. This map is $\phi^{-1}$-equivariant because $f_x^{-1}$ is $\phi_x^{-1}$-equivariant for each $x \in X$.}. Here is how we define $\xi: \cG \to  (P \times_X F(\cG)) \times_{G \rtimes \Aut(G)}G$: assume $\tilde{g} \in \cG_x$ for some $x \in X$; then, choose any isomorphism $\alpha: \cG_x \to G$, $q \in P$ and let $\tilde{g} \mapsto [q, \alpha^{-1}, \alpha(\tilde{g})]$. This is easily seen to be a well-defined map, as follows. Say $\tilde{g} \mapsto [p, \varphi^{-1}, \varphi(\tilde{g})]$. We know that there exists $(h, \psi) \in G \rtimes \Aut(G)$ such that $(q, \alpha^{-1})(h, \psi) = (q\psi(h), \alpha^{-1} \circ \psi) = (p, \varphi)$. Thus, in $(P \times_X F(\cG)) \times_{G \rtimes \Aut(G)}G$, $$[q, \alpha^{-1}, \alpha(\tilde{g})] = [q \psi(h), \alpha^{-1} \circ \psi, \psi^{-1} \circ \alpha(\tilde{g})] = [p, \varphi^{-1}, \varphi(\tilde{g})].$$ The inverse of $\xi$ is given by $\xi^{-1}: (P \times_X F(\cG)) \times_{G \rtimes \Aut(G)}G \to \cG, [p, \phi, g] \mapsto \phi(g)$. Now, we define the morphism $f: P \to {\left(P \times_X F(\cG) \right)}/{\Aut(G)}$ by $p \mapsto [p, \alpha]$ for $p \in P_x$, where $\alpha: G \to \cG_x$ is any Lie group isomorphism. This morphism, firstly, is well defined: if $\beta: G \to \cG_x$ is another isomorphism, then $[p, \alpha] = [(p, \alpha)(e, \alpha^{-1}\circ \beta)] = [p, \beta]$. Secondly, it is $\xi$-equivariant; for, if $p \in P_x$ and $g \in \cG_x$, then 
\begin{align*} pg \mapsto [pg, \alpha]_{\Aut(G)}  &= [(p, \alpha)(g, \id_G)]_{\Aut(G)} \\ &= [p, \alpha]_{\Aut(G)}[p, \alpha, g]_{G \rtimes \Aut(G)}\\ &= [p, \alpha]_{\Aut(G)} \xi(g).
\end{align*}
We now address (2). We start with a $(G \rtimes \Aut(G))$-principal bundle $R \to X$. By definition, $\Gamma \circ \Lambda (R) = (R/\Aut(G)) \times_X \cF(R \times_{G \rtimes \Aut(G)}G)$. To show that this $(G \rtimes \Aut(G))$-principal bundle is isomorphic to $R$, we only need to show that there is a $(G \rtimes \Aut(G))$-equivariant morphism from one to the other. Define $R \to (R/\Aut(G)) \times_X \cF(R \times_{G \rtimes \Aut(G)}G)$ by $r \mapsto ([r], \varphi)$, where, assuming $r \in R_x$ for some $x \in X$, $\varphi: G \to R_x \times _{G \rtimes \Aut(G)} G, g \mapsto [r, g]$. Finally, we see easily that this morphism of fibre bundles is $(G \rtimes \Aut(G))$-invariant. Indeed, say $(g, \psi) \in G \rtimes \Aut(G)$ and $r \in R_x$, and let $\varphi$ be as above; then observe that
\begin{align*}
r(g, \psi) &\mapsto ([r(g, \psi)], \varphi \circ \psi) \\
&= ([r(g, \id_G)], \varphi \circ \psi) \\
&= ([r]_{\Aut(G)}[r, g]_{G\rtimes \Aut(G)}, \varphi \circ \psi) \\
&= ([r]\varphi(g), \varphi \circ \psi) \\
&= ([r, \varphi])(g, \psi).
\end{align*}
Finally, verifying that these isomorphisms are natural is routine, and we conclude the proof of Theorem \ref{Murray correspondence} here. 
\\

The equivalence of the categories $\cR_GX$ and $\cP_GX$ allows us to carry a number of classical results from ordinary principal bundle theory over to the setting of generalised principal bundles. Of particular importance to us is the following deduction. Say we have $G$-group bundles $\cG_0$ and $\cG_1$ over a Hausdorff, second countable space $X$. Assume we are also given 1-cocycles $f = (f_{ij}: U_{ij} \to \cG_0)_{i, j \in I}$ and $g  = (g_{ij}: U_{ij} \to \cG_1)_{i, j \in I}$ of continuous sections of $\cG_0$ and $\cG_1$ respectively, with respect to an open cover $\cU = (U_i)_{i \in I}$  of $X$. Let $P_f \to X$ be the $\cG_0$-principal bundle defined by the cocycle $f$, and let $P_g \to X$ be the $\cG_1$-principal bundle defined by $g$. Assuming, as we may, that $\cU$ is an adequately fine open cover, $P_f$ (respectively, $P_g$) is specified by a pair of 1-cocycles $(f^{\alpha}, f^{\beta})$ (respectively, $(g^{\alpha}, g^{\beta})$) with respect to $\cU$, where $f^\alpha$ and $g^\alpha$ take values in $\Aut(G)$ and define the group bundles $\cG_0$ and $\cG_1$ respectively, and $f^\beta$ and $g^\beta$ take values in $G$ and define $P_f$ and $P_g$ respectively. (We see this expressly, without the need to invoke Theorem \ref{Murray correspondence}. Indeed, $f^\alpha$ and $g^\alpha$ are the transition functions of the group bundles $\cG_0$ and $\cG_1$ respectively; similarly, $f^\beta$ and $g^\beta$ specify the transition functions of the fibre bundles $P_f$ and $P_g$ respectively -- the fibres of $P_f$ and $P_g$ are right $G$-torsors, so any automorphism of such a fibre, upon trivialisation, is given by left multiplication by an element of $G$.) Now, assume that $f^\alpha$ is homotopic to $g^\alpha$ through such 1-cocycles with respect to $\cU$, and that $f^\beta$ is similarly homotopic to $g^\beta$. (A special case of this situation is when $\cG_0 = \cG_1$, and $f$ is homotopic to $g$ through 1-cocycles of continuous sections of $\cG_0$ with respect to $\cU$.) This is equivalent to the assumption that there exists a $G$-group bundle $\cG$ over the base $X \times I$, and a $\cG$-principal bundle $P \to X \times I$, such that the \textit{0-endpoint pair} over $X$ (that is, the object of $\cR_GX$ induced by restricting the base to $X \times \{0\}$) is isomorphic to $(P_f, \cG_0)$, and the \textit{1-endpoint pair} is isomorphic to $(P_g, \cG_1)$. By Theorem \ref{Murray correspondence}, the categories $\cR_G(X \times I)$ and $\cP_{G \rtimes \Aut(G)}(X \times I)$ are equivalent, so $(P, \cG)$ corresponds to a $(G \rtimes \Aut(G))$-principal bundle $R$ over $X \times I$. The $(G \rtimes \Aut(G))$-principal bundles $R_f$ and $R_g$ -- which are induced by restricting the base to $X\times \{0\}$ and $X \times \{1\}$ respectively -- evidently correspond, under the equivalence of the categories $\cR_GX$ and $\cP_{G \rtimes \Aut(G)}X$, to $(P_f, \cG_0)$ and $(P_g, \cG_1)$ respectively. But, from classical results (see, for example, \cite{husemoller1994fibre}), we know that the ordinary principal bundles $R_f$ and $R_g$ are isomorphic, which in turn means that $(P_f, \cG_0)$ and $(P_g, \cG_1)$ are isomorphic in the $\cR_GX$. The following corollary summarises this:

\begin{corollary} \label{cor: htpy invariance}
Let $r_0: X \times I \to X \times I$ be the map given by $(x, t) \mapsto (x, 0)$, and let $r_1: X \times I \to X \times I$ be given by $(x, t) \mapsto (x, 1)$. If $(P, \cG)$ is an object of $\cR_G(X \times I)$, then $(r_0^*(P), r_0^*(\cG))$ and $(r_1^*(P), r_1^*(\cG))$ are isomorphic objects of $\cR_G(X)$.
\end{corollary}
Here is a special case of the above result:
\begin{corollary} \label{cor: homotopic implies iso}
Let $\cG$ be a $G$-group bundle and let $f$ and $g$ be continuous 1-cocycles of the sheaf of sections of $\cG$, which are homotopic, through such 1-cocycles, over a common refinement of the open cover of $X$ on which they are defined. Then, the $\cG$-principal bundles defined by $f$ and $g$ are topologically isomorphic.
\end{corollary}

\end{definition}

\subsection{The parametric Oka-Grauert principle for generalised principal bundles} 

The central foundation for our main result is the parametric Oka-Grauert principle for generalised principal bundles (Theorem \ref{parametric Cartan} below). This theorem was close to the surface of the classical works of Cartan and Grauert \cite{cartan1958espaces,grauert1958}, though not explicitly formulated during that period. We note that the theorem is subsumed by more general results in modern, Gromov-style Oka theory, proven using different methods (see \cite[Theorem 5.4.4]{forstnerivc2011stein}). The proof we provide below uses only the methods of Cartan and Grauert, offering a more direct approach to this important special case.

We first set up the technical foundation. 
Assume $X$ is a Stein space, and let $E$ be a holomorphic Lie group bundle with base $X$. Following Cartan \cite{cartan1958espaces}, we denote by $\cE^a$ the sheaf of holomorphic sections of $E$, and $\cE^c$ the sheaf of continuous sections of $E$.

Let $C$ be a compact Hausdorff space, and let $N \subset H \subset C$ be closed subspaces. By an $(N, H, C)$-map in $\cE^c(U)$, we mean a continuous family $\varphi_t$ of sections of $\cE^c$ over $U$, parametrised by $t \in C$, such that 
\begin{enumerate}
\item for each $t \in N$, $\varphi_t$ is the identity section (that is, assumes the value of the identity element in each fibre), and
\item for each $t \in H$, $\varphi_t \in \cE^a(U)$.
\end{enumerate}
We may write an $(N, H, C)$-map in the form $C \times U \to E$. Denote by $\cF$ the sheaf of the $(N, H, C)$-maps in $\cE^c$; that is, for each open set $U \subset X$, $\cF(U)$ is a topological group, equipped with the compact-open topology, comprising the $(N, H, C)$-maps in $\cE^c(U)$. It is straightforward to verify that $\cF(U)$ is a topological group for any open set $U \in \cU$ and that $\cF$ is a sheaf.
\begin{theorem}\textnormal{\cite{cartan1958espaces}} \label{theorem: main theorem} 
  Let $C$ be a compact Hausdorff space, and let $N \subset H \subset C$ be closed subspaces. Assume that $N$ is a weak deformation retract of $C$ -- that is, there exists a continuous map $r : C \times I \to C$ satisfying $r(x,0) = x$ and $r(x,1) \in N$ for all $x \in C$, with $r(x,1) = x$ when $x \in N$. Then, 
\begin{enumerate}
\item[\textnormal{(1)}] the topological group $H^0(X, \cF)$ is path connected,
\item[\textnormal{(2)}] if $U \subset X$ is an open, holomorphically convex subset, then the image of the map $H^0(X, \cF) \to H^0(U, \cF)$ is dense in $H^0(U, \cF)$, and
\item[\textnormal{(3)}] $H^1(X, \cF) = 0$.
\end{enumerate}
\end{theorem}

The parametric Oka-Grauert principle for generalised principal bundles is an extension of the following classical (non-parametric)  principle, which appeared in the Cartan-Grauert period in \cite{cartan1958espaces,grauert1958}: 
\begin{theorem} \label{Cartan's thm 1}
  Let $P$ be an $E$-principal bundle over a base $X$ which is a Stein space. Then every continuous section of $f$ of $P$ is homotopic to a holomorphic section of $P$. 
  \end{theorem}

  We may now state and prove the parametric principle. 

  \begin{theorem} \label{parametric Cartan} 
    Let $X$ be a reduced Stein space, $E$ a holomorphic Lie group bundle with base $X$, and $P$ be a holomorphic $E$-principal bundle over $X$. The inclusion $\Gamma_\cO(P, X) \hookrightarrow \Gamma_\cC(P,X)$ of the space $\Gamma_\cO(P, X)$ of holomorphic sections of $P$ into the space $\Gamma_\cC(P,X)$ of continuous sections of $P$ is a weak homotopy equivalence.
    \end{theorem}
    \begin{proof} Let $f_0 \in \Gamma_\cO(P, X)$ be arbitrary. We already know (cf. Theorem \ref{Cartan's thm 1}) that the inclusion $\Phi: \Gamma_\cO(P, X) \hookrightarrow \Gamma_\cC(P,X)$ induces a surjection of path components. We may assume that $\Gamma_\cO(P, X)$ is nonempty: if there were no holomorphic sections, then there would be no continuous sections. By identifying $f_0$ with the identity section of $E$, we may view any section of $P$ as a section of the structure group bundle $E$. This implies that $\Gamma_\cO(P, X)$ and $\Gamma_\cC(P, X)$ are topological groups, and therefore the sets $\pi_0(\Gamma_\cO(P, X))$ and $\pi_0(\Gamma_\cC(P, X))$ are themselves groups. Thus, to show that the homomorphism \[\pi_n(\Gamma_\cO(P, X), f_0) \to \pi_n(\Gamma_\cC(P, X), f_0)\] is injective for each $n \geq 0$, we only need to show that the kernel vanishes. 
    
    We will prove at the same time that $\Phi$ induces a $\pi_n$-monomorphism and a $\pi_{n+1}$-epimorphism for all $n \geq 0$. Let $B^n$ be the closed unit ball in $\RR^n, n \geq 1$, and let $b_0 \in \partial B^n$ be a given basepoint. Assume we have a continuous map $\varphi: B^n \to \Gamma_\cC(P, X)$ such that $\varphi (\partial B^n) \subset \Gamma_\cO(P, X)$ and $\varphi(b_0) = f_0$. We show that there exists a continuous map $\tilde{\varphi}: B^n \times I \to \Gamma_\cC(P, X)$ such that 
    \begin{enumerate}
    \item $\tilde{\varphi}(\cdot, 0) = \varphi$,
    \item $\tilde{\varphi}(b_0, \cdot) = f_0,$ and 
    \item the image of $\tilde{\varphi}(\cdot, 1)$ lies in $\Gamma_\cO(P, X)$.
    \end{enumerate}
    Define a map $\gamma: B^n \to \Gamma_\cC(P, X)$ by $b \mapsto \varphi(b)\cdot f_0^{-1}$. Viewing sections of $P$ over $X$ as sections $E$ over $X$, observe that $\gamma$ is an $(N, H, C)$-map for $C = B^n, H = \partial B^n$ and $N = \{b_0\}$. By (1) of Theorem \ref{theorem: main theorem}, $H^0(X, \cF)$ is path connected, which implies that there exists a continuous map $\tilde{\gamma}: B^n \times I \to \Gamma_{\cC}(P, X)$ such that $\tilde{\gamma}(\cdot, 0) = \gamma$, and $\tilde{\gamma}(b_0, \cdot) = \tilde{\gamma}(\cdot, 1) = e$ (the identity section of $E$). Define \[\tilde{\varphi}: B^n \times I \to \Gamma_\cC(P, X), \,(b, t) \mapsto \tilde{\gamma}(b, t) \cdot f_0.\] This map satisfies the required conditions and concludes the proof.
    \end{proof}
\section{Homotopy-theoretic results} \label{sec:homotopy}

The $S^{-1}S$ construction, reported by D.~Grayson in \cite{grayson2006higher} based on D.~Quillen's work from the early 1970s, is a significant generalisation of the Grothendieck group completion process to symmetric monoidal categories. Quillen used the $S^{-1}S$ construction to define the higher algebraic K-theory of a symmetric monoidal category $S$. We require a simplicially enriched version of the $S^{-1}S$ construction to handle the topological (or simplicial) information about hom-sets. This leads us to work with simplicial categories; that is, categories enriched over sSet, the category of small simplicial sets. A simplicial category may be equivalently interpreted as a simplicial object in the category Cat of small categories, with the additional condition that the simplicial operators induce the identity morphism on the objects of the categories associated with each simplicial level.

Our focus lies in particular with symmetric monoidal categories enriched over sSet -- which serve as inputs to the enriched $S^{-1}S$ construction, defined below -- and in studying the behaviour of the enriched $S^{-1}S$ construction under weak equivalences of such categories. Specifically, we present conditions under which weak equivalences are preserved, culminating in a series of results that form the technical foundation for the subsequent (and main) section of this paper. 

\subsection{Simplicial symmetric monoidal categories} \label{ssmc}
 We recall the notion of a symmetric monoidal category: a category $S$ with a functor $\square: S \times S \to S$, a unit object $e$, and natural isomorphisms -- which satisfy certain coherence conditions -- that establish associativity, commutativity, and identity laws (see \cite{mac2013categories} for more context and precise details). 

\begin{definition}
A simplicial symmetric monoidal category is a simplicial category $S$ with a \textit{simplicial} symmetric monoidal functor $\square: S \times S \to S$. Expressly, given objects $a, b, c, d \in S$, the induced map on hom-sets, $S(a, b) \times S(c, d) \to S(a \, \square \, c, b \, \square \, d)$, is a simplicial map. 
\end{definition}

We will soon see illustrative examples of the above. First, we establish some notation: for a simplicial category $S$, its category of components $\pi_0 S$ has the same objects as $S$, with hom-sets $\pi_0 S(x,y) = \pi_0(S(x,y))$ for objects $x,y$ in $S$. Composition in $\pi_0 S$ is induced from the composition in $S$. 

Throughout this paper, we use Bergner's model category structure \cite{bergner2007model} on simplicial categories, where weak equivalences are taken to mean \textit{Dwyer-Kan equivalences}, first described by Dwyer and Kan in \cite{dwyer1980function}. In our context of simplicial symmetric monoidal categories, we need to further ask that the Dwyer-Kan equivalences respect the symmetric monoidal structure. Here is the definition we use:
\begin{definition} \label{simp dwyer kan}
    If $(S_1, \square_1)$ and $(S_2, \square_2)$ are simplicial symmetric monoidal categories, an sSet-enriched functor $F: S_1 \to S_2$ is a \textit{weak equivalence} of simplicial symmetric monoidal categories if 
\begin{enumerate}
\item[(1)] $F$ is an monoidal functor,
\item[(2)] if $a, b \in \mathrm{ob}(S_1)$, then the induced map $S_1(a, b) \to S_2(F(a), F(b))$ is a weak equivalence of simplicial sets, and 
\item[(3)] the induced functor between the categories of components, $\pi_0f: \pi_0S_1 \to \pi_0S_2$, is an equivalence of categories. 
\end{enumerate}
\end{definition}

\noindent Before moving on to the details of the group completion construction, we provide two examples that are illustrative and particularly pertinent to us. 
\begin{example} \label{ex 1}
   Consider a homotopy equivalence $X \rightarrow Y$ of paracompact topological spaces. Let $i\mathrm{Vect}^\mathcal{C}_X$ and $i\mathrm{Vect}^\mathcal{C}_Y$ denote the categories of complex topological vector bundles on $X$ and $Y$, respectively, with morphisms restricted to vector bundle isomorphisms. We may enrich these categories in topological spaces by equipping each hom-set -- say, from $V$ to $W$ -- with the topology of $\Gamma(Q)$, the space of sections of the isomorphism bundle $Q$ from $V$ to $W$ (see \cite{cartan1958espaces}). These categories form topological symmetric monoidal categories under direct sum. The pullback operation induces a topological functor  $(\mathrm{Vect}^\mathcal{C}_Y, \oplus) \to (\mathrm{Vect}^\mathcal{C}_X, \oplus)$ given by  $(E \to Y) \mapsto (E \times_Y X \to X)$.

Since the morphism spaces are spaces of sections of principal bundles, the homotopy invariance theorem for fibre bundles (cf. Corollary \ref{cor: htpy invariance}) implies that the pullback functor satisfies analogues of the Dwyer-Kan conditions in the category of topologically enriched categories. We can then apply the singular functor to the morphism spaces to convert these categories into simplicial categories, making the functor a weak equivalence of simplicial symmetric monoidal categories. 
\end{example}
The following extended example is an Oka principle; in fact, it is an an elementary version of the main result of this paper. 
\begin{example} \label{ex 2}
Let $X$ be a reduced Stein space. Consider the categories $i\mathrm{Vect}^\mathcal{C}_X$ and $i\mathrm{Vect}^\mathcal{O}_X$ whose objects are, respectively, the complex topological and holomorphic vector bundles over $X$. The morphisms in each of the categories are isomorphisms of vector bundles, topological and holomorphic respectively. As is classically known, any object in $i\mathrm{Vect}^\mathcal{C}_X$ (or $i\mathrm{Vect}^\mathcal{O}_X$) with constant rank $n$ is determined by an associated principal bundle over $X$ with structure group $GL_n\mathcal\CC$. Moreover, these categories come equipped with a symmetric monoidal functor given by the direct sum ($\oplus$) of vector bundles. As the symmetric monoidal structure can tie together vector bundles with varying ranks on different connected components, we may restrict our attention to bundles of constant rank.

As in the the previous example (Ex. \ref{ex 1}), the categories $i\mathrm{Vect}^\mathcal{C}_X$ and $i\mathrm{Vect}^\mathcal{O}_X$ both admit enrichments over topological spaces: indeed, isomorphisms between two objects, say from $V$ to $W$, are sections of the isomorphism bundle from $V$ to $W$ (see \cite[\S3]{cartan1958espaces}). Direct sum is compatible with this enrichment. Moreover, in the holomorphic category $i\mathrm{Vect}^\mathcal{O}_X$, each isomorphism bundle admits a holomorphic structure. 

Now, consider the functor $\iota: (i\mathrm{Vect}^\mathcal{O}_X, \oplus) \to (i\mathrm{Vect}^\mathcal{C}_X, \oplus)$, induced by inclusion, between these topological symmetric monoidal categories.
If $V$ and $W$ are objects in $i\mathrm{Vect}^\mathcal{O}_X$, the parametric Oka principle for generalised principal bundles (Theorem \ref{parametric Cartan}) implies that the inclusion $\Gamma_\cO(Q, X) \to \Gamma_\cC(Q, X)$ of the space of holomorphic sections into the space of continuous sections of the isomorphism bundle $Q$ (from $V$ to $W$) is a weak homotopy equivalence. This is the topological analogue of condition (2) of Definition \ref{simp dwyer kan}. 

The functor $\iota$ also meets the topological analogue of condition (3); that is, the induced functor between the categories of components, $\pi_0\iota: \pi_0(i\mathrm{Vect}^\mathcal{O}_X) \to \pi_0(i\mathrm{Vect}^\mathcal{C}_X)$, is an equivalence of categories. Indeed, Theorem \ref{parametric Cartan} already establishes that the functor is full and faithful. We get essential surjectivity from Theorem B of \cite{cartan1958espaces} -- which states that on every {topological}  (generalised) principal bundle, there is a compatible {holomorphic} structure. 
By applying the singular functor Sing to the morphism spaces, these categories may be enriched in simplicial sets. As the category of simplicial sets is Quillen equivalent to the category of topological spaces, Sing($\iota$) is a weak equivalence of simplicial symmetric monoidal categories.
\end{example}

\subsection{The group completion construction} \label{gp completion}

Our first aim is to upgrade Quillen's discrete $S^{-1}S$ construction to a simplicially enriched version.  We do this by applying the $S^{-1}S$ construction to each simplicial level, assembling them together, and ensuring that simplicial structures are preserved. Let $(S, \square)$ be a simplicial symmetric monoidal  category.  To start, we make some assumptions about $S$. 

\begin{definition} \label{def: good}
    Let $S$ be a small simplicial symmetric monoidal category; that is, it has a set of objects and simplicial sets of morphisms. We say that $S$ is \textit{good} when         
\begin{enumerate}
    \item every morphism is an isomorphism, and
    \item for each $x \in \mathrm{ob}(S)$, the simplicial functor $x \square -: S \to S$ is faithful.
\end{enumerate} 
\end{definition}
The original $S^{-1}S$ construction \cite{grayson2006higher} for a discrete symmetric monoidal category $S$ uses the second condition above -- that  \textit{translations are faithful in $S$} -- to prove that $BS^{-1}S = |NS^{-1}S|$ is indeed a group completion for $BS = |NS|$. As we produce a simplicially enriched version of the same construction in the following section, we make this assumption as well. Moreover, as we will see, assuming that translations are faithful provides greater control over the morphisms in the simplicial category $S^{-1}S$, and is a key assumption the proof of Theorem \ref{main prop}.
\begin{example}
    Recall the topologically enriched category $i\mathrm{Vect}^\mathcal{O}_X$ from Example \ref{ex 2}, whose objects are holomorphic vector bundles over $X$, and whose morphism spaces  -- say from $V$ to $W$ -- are given by $\Gamma_\cO(Q, X)$, the space of holomorphic sections of the isomorphism bundle $Q$ from $V$ to $W$. Equipped with the direct sum $\oplus$, this category is symmetric monoidal (evidently, $\oplus$ is compatible with the topological enrichment). As in Example \ref{ex 2}, we may apply the singular functor to the morphism spaces to obtain a simplicial symmetric monoidal category, and it is simple to verify that this category is good. 
\end{example}

\begin{definition}
    For a simplicial symmetric monoidal category $(S, \square)$, we obtain a new simplicial symmetric monoidal category $S^{-1}S$.   The objects of $S^{-1}S$ are pairs $(x, y)$ of objects in $S$, and the simplicial hom-sets are defined as follows. For objects $(x, y)$ and $(z, w)$, the $n$th simplicial level of the hom-set $S^{-1}S((x, y), (z, w))$ is \[ S^{-1}S((x, y), (z, w))_n = \faktor{\{ (s, f, g) \vert s \in \mathrm{ob}(S), \, f \in S(s \square x,z)_n, \, g \in S(s \square y, w)_n \}}{\sim_n},\] where two $n$-simplices $(s, f, g)$ and $(s', f', g')$ are equivalent under $\sim_n$ precisely when there is an isomorphism $\alpha: s \to s'$ in $S(s, s')_n$ so that $f = f' \circ (\alpha \square 1_x)$ and $g = g' \circ (\alpha \square 1_x)$, where $1_x \in S(x, x)_n$ and $1_y \in S(y, y)_n$ are the degenerate identity $n$-simplices. When $S$ is good, note that the morphisms of $S^{-1}S$ determine $s, f,$ and $g$ up to unique isomorphism. The functor $\tilde{\square}$ is given by $(x_1, x_2) \tilde{\square} (y_1, y_2) = (x_1 \square y_1, x_2 \square y_2)$.
\end{definition}

We must now verify that this definition is valid -- that these $n$-simplices assemble into simplicial sets of morphisms, and give $S^{-1}S$ the structure of a simplicial category.
Firstly, the sets $S^{-1}S((x, y),(z, w))_n, n\geq 0,$ form a simplicial set $S^{-1}S((x, y),(z, w))$ with face and degeneracy maps induced by those from $S(s \square x, z)$ and $S(s \square y, w)$. Evidently, these maps are well defined under the equivalence relation because $S$ is a simplicial category -- in particular, because composition in $S$ commutes with the face and degeneracy maps.

Secondly, composition in  $S^{-1}S$ is defined as follows: given $n$-simplices $(s, f, g) \in S^{-1}S(x_1, x_2)$ and $(t, a, b) \in S^{-1}S(x_2, x_3)$, let \[(t, a, b) \circ (s, f, g)  = (t \square s, a \circ (t \square f), b \circ (t \square g)).\] We observe immediately that this definition of composition respects the simplicial structures of the hom-sets (that is, it is compatible with face and degeneracy maps), because $s \square -$ is a simplicial functor for all $s \in S$, and composition in $S$ respects the simplicial structures of hom-sets in $S$. We also note that the composition is associative, because $\square$ is associative up to coherent natural isomorphism, and the equivalence relation $\sim_n$ allows us to identify factors up to isomorphism; that is, the associator $\alpha: (t \square s) \square r \to t \square (s \square r)$ in $(S, \square)$ makes the $n$-simplices 
\[[(t, a, b) \circ (s, f, g)] \circ (r, u, v) = ((t \square s) \square r, (a \circ(t \square f))\circ (t \square s) \square u, (b \circ(t \square g))\circ (t \square s) \square v)\] and \[(t, a, b) \circ [(s, f, g) \circ (r, u, v)] = (t \square (s \square r), a \circ (t \square (f \circ (s \square u))), b \circ (t \square (g \circ (s \square v))))\] equivalent in $S^{-1}S$. 

The final piece we need for $S^{-1}S$ to be a simplicial category is identity morphisms. In $S^{-1}S((x, y)(x, y))_n$, we let the identity $n$-simplex be $(e, 1_x, 1_y)$ where $e$ is the unit object in $(S, \square)$, and $1_x$ and $1_y$ are, as above, the degenerate identity $n$-simplices in $S(x, x)_n$ and $S(y, y)_n$ respectively.

So far, we have constructed a simplicial category $S^{-1}S$ out of $S$. Finally, we consider the symmetric monoidal product $\tilde{\square}$. Observe that the induced functor \[S^{-1}S((a_1, a_2), (b_1, b_2)) \times S^{-1}S((c_1, c_2), (d_1, d_2)) \to S^{-1}S((a_1 \square c_1, a_2 \square c_2), (b_1 \square d_1, b_2 \square d_2))\] is monoidal, and is compatible with the face and degeneracy maps of $S^{-1}S$, because $\square$ respects the simplicial structures of the hom-sets of $S$. We have thus constructed out of $(S, \square)$ a simplicial symmetric monoidal category $(S^{-1}S, \tilde{\square})$, which comes with a simplicial symmetric monoidal functor 
$S \to S^{-1}S$ that maps $x \in \mathrm{ob}(S)$ to $(e, x) \in \mathrm{ob}(S^{-1}S)$ and $g \in S(x, y)$ to $(e, 1_e, g) \in S^{-1}S((e, x),(e,y))$.
 
Finally, as below, the construction is functorial (Proposition \ref{prop: Lambda}). The proof is straightforward, and we omit it. 
\begin{definition}
We denote by $\mathrm{SymMonCat}_\Delta$ the category whose objects are the small simplicial symmetric monoidal categories, and whose morphisms $(S_1, \square_1) \to (S_2, \square_2)$ are sSet-enriched monoidal functors. We denote by $\mathrm{SymMonGpd}_\Delta^\dagger$ the full subcategory of good simplicial symmetric monoidal categories.
\end{definition}

\begin{proposition} \label{prop: Lambda}
    There is a well-defined functor $\Lambda: \mathrm{SymMonGpd}_\Delta^\dagger \to  \mathrm{SymMonCat}_\Delta$, defined as follows: 
    
    \begin{enumerate}
        \item on objects $(S, \square)$ by
            $\Lambda(S, \square) = (S^{-1}S, \tilde{\square})$, and
        
        \item on morphisms $\phi: (S_1, \square_1) \to (S_2, \square_2)$ by
            $\Lambda(\phi): (S_1^{-1}S_1, \tilde{\square}_1) \to (S_2^{-1}S_2, \tilde{\square}_2)$, which sends
            \begin{itemize}
                \item $(x, y) \in \mathrm{ob}(S_1^{-1}S_1)$ to $(\phi(x), \phi(y))\in \mathrm{ob}(S_2^{-1}S_2)$, and 
                \item $[s, f, g] \in \mathrm{mor}(S_1^{-1}S_1)$ to $[\phi(s), \phi(f), \phi(g)] \in \mathrm{mor}(S_2^{-1}S_2)$.
            \end{itemize}
    \end{enumerate}
    \end{proposition}

    \subsection{Weak equivalences and $S^{-1}S$}

    Our main interest is now to investigate how this construction behaves under weak equivalences. In particular, if we start with a weak equivalence $F: (S_1, \square_1) \to (S_2, \square_2)$ of good simplicial symmetric monoidal categories, what do we glean about the induced functor $(S_1^{-1}S_1, \tilde{\square}_1) \to (S_2^{-1}S_2, \tilde{\square}_2)$? We might expect $(S_1^{-1}S_1, \tilde{\square}_1) \to (S_2^{-1}S_2, \tilde{\square}_2)$ to be a weak equivalence also, and we find that this is indeed the case. We spell out the details below in Theorem \ref{main prop}.
    
    We first present an ingredient lemma. If $X$ is a simplicial set, and $G$ is a simplicial group, we say that $G$ \textit{acts on} $X$ if $G_n$ acts on each $X_n$ in a manner compatible with the face and degeneracy operators. The orbit spaces $X_n/G_n$ assemble to form a simplicial set, which we denote by $X/G$. (The simplicial structure of $X/G$ is inherited from $X$, and this structure is well defined because the action of $G$ respects the face and degeneracy operators of $X$ -- the details are straightforward to verify.)
    \begin{lemma} \label{prep l}
        Let $X$ and $Y$ be fibrant simplicial sets, along with free actions $F \times X \to X$ and $G \times Y \to Y$ by simplicial groups $F$ and $G$ respectively. Assume that there are weak equivalences $X \to Y$ and $F \to G$ of simplicial sets such that the diagram 
        \[\begin{tikzcd}
            F \times X \arrow{r} \arrow[swap]{d} & X \arrow{d} \\
            G \times Y \arrow{r} & Y
            \end{tikzcd}
            \] commutes. Then the induced map $X/F \to Y/G$ of simplicial sets is also a weak equivalence.
    \end{lemma}
    
    \begin{proof}
    Since the actions $F \times X \to X$ and $G \times Y \to Y$ are free, the quotient maps $X \to X/F$ and $Y \to Y/G$ are principal fibrations. Moreover, as $X$ and $Y$ are fibrant, these fibrations each give rise to a long exact sequence of homotopy groups, with a chain map between them:
    \[
        \begin{tikzcd}[arrows=to]
            \cdots \arrow[r] & \pi_{n+1}(X/F) \arrow[r] \arrow[d] & \pi_{n}(F) \arrow[r] \arrow[d] & \pi_{n}(X) \arrow[r] \arrow[d] & \pi_{n}(X/F) \arrow[r] \arrow[d] & \cdots \\
            \cdots \arrow[r] & \pi_{n+1}(Y/G) \arrow[r] & \pi_{n}(G) \arrow[r] & \pi_{n}(Y) \arrow[r] & \pi_{n}(Y/G) \arrow[r]  & \cdots
        \end{tikzcd} 
    \] 
    By assumption, the maps $\pi_n(F) \to \pi_n(G)$ and $\pi_n(X) \to \pi_n(Y)$ are isomorphisms for all $n \geq 1$ and bijections at $n=0$. The five-lemma (rather, its non-abelian variant, with the same diagram chase as the classical) now implies that $\pi_{n}(X/F) \to \pi_{n}(Y/G)$ is an isomorphism for all $n \geq 1$. That $\pi_{0}(X/F) \to \pi_{0}(Y/G)$ is an epimorphism immediately follows from the fact that $\pi_0(X) \to \pi_0(Y)$ is. Finally, verifying that $\pi_{0}(X/F) \to \pi_{0}(Y/G)$ is a monomorphism is done through a simple diagram chase argument. (We omit the details, which are routine.) 
    % Since $\pi_{0}(Y/G)$ and $\pi_{0}(X/F)$ in this context are pointed sets, denote by $[*_Y]$ and $[*_X]$ their respective basepoint components. We need only to show that the preimage of $[*_Y]$ in $\pi_0(X/F)$ comprises exactly one element -- this is because the choice of basepoint is arbitrary. Assume that $[\bar{x}] \in \pi_0(X/F)$ maps to $[*_Y] \in \pi_0(Y/G)$. As $\pi_0(X) \to \pi_0(X/F)$ is surjective, take $[x] \in \pi_0(X)$ in the preimage of $[\bar{x}]$. Now, since the diagram
    % \[
    % \begin{tikzcd}
    % \pi_0(X) \arrow[r] \arrow[d] & \pi_0(X/F) \arrow[d] \\
    % \pi_0(Y) \arrow[r] & \pi_0(Y/G)
    % \end{tikzcd}
    % \] is commutative, the image $[y] \in \pi_0(Y)$ of $[x]$ maps to the basepoint $[*_Y]$ in $\pi_0(Y/G)$. The long exact sequences of homotopy groups above are, in particular, exact in the $\pi_0$ portions (as pointed sets), so there exists $[g] \in \pi_0(G)$ in the preimage of $[y]$ under $\pi_0(G) \to \pi_0(Y)$. This corresponds to a unique element $[f]$ of $\pi_0(F)$, which maps to $[x] \in \pi_0(X)$, since the diagram 
    % \[
    % \begin{tikzcd}
    % \pi_0(F) \arrow[r] \arrow[d] & \pi_0(X) \arrow[d] \\
    % \pi_0(G) \arrow[r] & \pi_0(Y)
    % \end{tikzcd}
    % \] is commutative and the vertical maps are isomorphisms. Finally, the exactness at $\pi_0(X)$ (of the long exact sequence induced by $F \to X \to X/F$) tells us that $[x]$ maps to the basepoint component in $\pi_0(X/F)$ -- that is, $[*_X] = [\bar{x}]$ as we sought. 
    \end{proof}
    
    With this, we can now proceed to the main result of this section. 
    \begin{theorem} \label{main prop}
        Let $(S_1, \square_1)$ and $(S_2, \square_2)$ be good simplicial symmetric monoidal categories. Let $F: (S_1, \square_1) \to (S_2, \square_2)$ be a monoidal functor which is also a weak equivalence. That is, $F$ satisfies the following:
        \begin{enumerate}
            \item \label{strict} for all $a, b \in \mathrm{ob}(S_1)$, $F(a \square_1 b) = F(a) \square_2 F(b)$,
            \item \label{hom set} each induced map of hom-sets $S_1(a, b) \to S_2(F(a), F(b))$ is a weak equivalence of simplicial sets, and
            \item \label{ess sur}for all $c \in \mathrm{ob}(S_2)$, there exists $d \in \mathrm{ob}(S_1)$ such that $F(d)$ and $c$ are isomorphic in the category $\pi_0 S_2$ of components of $S_2$.
        \end{enumerate}
    Then, the induced simplicial symmetric monoidal functor $\tilde{F}: (S_1^{-1}S_1, \tilde{\square}_1) \to (S_2^{-1}S_2, \tilde{\square}_2)$ satisfies the corresponding versions of conditions above -- which we call (\ref{strict}*), (\ref{hom set}*), and (\ref{ess sur}*) -- making it a weak equivalence of simplicial symmetric monoidal categories.
    \end{theorem}
    \begin{remark} Note that, while the third condition in Definition \ref{simp dwyer kan} requires the induced functor $\pi_0F: \pi_0S_1 \to \pi_0S_2$ to be an equivalence of categories, we already get fullness and faithfulness from $S_1(a, b) \to S_2(F(a), F(b))$ being a bijection at the level of path components; so we only need essential surjectivity in (\ref{ess sur}*).
    \end{remark}
    \begin{proof}
        In this proof, we assume that $F$ is a strict monoidal functor; that is, $F(a \square_1 b) = F(a) \square_2 F(b)$ for all objects $a,b \in \mathrm{ob}(S_1)$. This is purely for readability and conciseness of the arguments -- there is no deeper mathematical reason for it. Moreover, in our applications we only work with strict monoidal functors. The proof can be extended to (strong) monoidal functors with some technical adjustments, and without additional conceptual difficulty.
    
        The proofs of (\ref{strict}*) and (\ref{ess sur}*) are straightforward. For (\ref{strict}*), observe simply that for $(a, b), (c, d) \in \mathrm{ob}(S_1^{-1}S_1)$,
    \begin{align*}
        \tilde{F}((a,b)\tilde{\square}_1(c, d)) 
        &= (F(a \square_1 c), F(b \square_1 d)) \\
        &= (F(a)\square_2F(c), F(b)\square_2F(d)) \\
        &= \tilde{F}(a, b) \tilde{\square}_2 \tilde{F}(c, d).
    \end{align*}
    For (\ref{ess sur}*), consider  an object $(c, d)$ in  $S_2^{-1}S_2$. We seek an object $(a,b)$ in $S_1^{-1}S_1$ such that $\tilde{F}(a, b)$ is isomorphic to $(c,d)$ in the category $\pi_0\,S_2^{-1}S_2$ of components of $S_2^{-1}S_2$. By assumption, there exist isomorphisms $F(a) \to c$ and $F(b) \to d$ in $S_2$ (note: not just in the category of components, as $S_2$ is assumed to be good), which induce a map in $S_2^{-1}S_2$ given by the equivalence class of $(F(a), F(b)) \to (c, d)$; and, by its very construction, this map satisfies (\ref{ess sur}*).
    
    Finally, we address (\ref{hom set}*). For this, we reformulate the definition of the morphisms of the simplicial category $S^{-1}S$ in a way that simplifies the following arguments. Replace $S_1$ by a full skeletal subcategory $\tilde{S}_1$ that contains only one object in each isomorphism class; and replace $S_2$ by the full subcategory $\tilde{S}_2$ where the objects are given by the image under $F$ of the objects of $\tilde{S}_1$. (We may do this because the categories are good.) We make two observations: first, by assumption (\ref{ess sur}), every isomorphism class of $S_2$ is represented by an object in $\tilde{S}_1$; second, by the $\pi_0$-mono case of assumption (\ref{hom set}), this representative is unique. Therefore, as $F|_{\tilde{S}_1}$ is a bijection on objects, we may relabel the objects of $\tilde{S}_2$ so that $\mathrm{ob}(\tilde{S}_1) = \mathrm{ob}(\tilde{S}_2)$. Since $F$ is strictly symmetric monoidal,  we have the following commutative diagram of symmetric monoidal categories, 
    \[\begin{tikzcd}
    (\tilde{S}_1, \tilde{\square}_1) \arrow{r}{F|_{\tilde{S}_1}} \arrow[swap]{d}{\iota_1} & (\tilde{S}_2, \tilde{\square}_2) \arrow{d}{\iota_2} \\
    (S_1, \square_1) \arrow{r}{F} & (S_2, \square_2)
    \end{tikzcd},
    \]
    so we abuse notation for convenience and simply write $(S_1,\square_1)$ and $(S_2, \square_2)$ to denote the skeletal subcategories $(\tilde{S}_1, \tilde{\square}_1)$ and $(\tilde{S}_2, \tilde{\square}_2)$.  For objects $(a, b), (c, d)$ in $S_i \times S_i$ ($i = 1$ or $2$), we may now describe $S_i^{-1}S_i((a, b), (c, d))$ as the disjoint union
    \[\coprod_{x \in \mathrm{ob}(S_i)}\faktor{S_i \times S_i((x \square a, x \square b), (c, d))}{\mathrm{Aut}_{S_i}(x)}, 
        \]
    where  the action of a group element $ \varphi \in \mathrm{Aut}_{S_i}(x)$ on $(f, g) \in S_i \times S_i((x \square a, x \square b), (c, d))$ is given by $(f, g) \circ (\varphi \square 1_{(a, b)})$. Observe now that this action is free. Indeed, if $(f, g) \circ (\varphi \square 1_{(a, b)}) = (f, g) \circ (\psi \square 1_{(a, b)})$ then $\varphi \square 1_{(a, b)} = \psi \square 1_{(a, b)}$ as $S_i \times S_i$ is a groupoid; which in turn implies, as the translation functor is faithful (from $S_i$ being good), that $\varphi = \psi$.
    
    From Lemma \ref{prep l}, we get a weak homotopy equivalence of simplicial sets, \[ \faktor{S_1 \times S_1((x \square a, x \square b), (c, d))}{\mathrm{Aut}_{S_1}(x)} \, \rightarrow  \,\faktor{S_2 \times S_2((x \square a, x \square b), (c, d))}{\mathrm{Aut}_{S_2}(x)}.\]
    The final piece to this proof comes down to observing, using Ken Brown's lemma (see \cite{hovey2007model}) and the fact that every object in the model category of simplicial sets is cofibrant, that the class of weak equivalences of simplicial sets is closed under coproduct. In particular, the map $S_1^{-1}S_1((a, b), (c, d)) \to S_2^{-1}S_2((a, b), (c, d))$ is a weak equivalence.
    \end{proof}
    
    An important application of this result follows. Recall that the nerve $NC$ of a category $C$ is a simplicial set with $n$-simplices given by sequences $X_n \to X_{n-1} \to \cdots \to X_1 \to X_0$ of composable arrows in $C$. Now, if $C$ is a small category enriched in simplical sets, we obtain a bisimplicial set by applying the nerve functor to each simplicial level. In this case, by a widely accepted (slight) abuse of notation, we call the diagonal of this bisimplical set the nerve of $C$, and denote it by $NC$. 
    It is well known (see, for example, \cite[Section IV.3]{weibel2013k}) that if $f: X \to Y$ is a map of bisimplicial sets such that the induced simplicial map $X_{i, *} \to Y_{i, *}$ is a weak equivalence for all $i \geq 0$, then $\mathrm{diag}(X) \to \mathrm{diag}(Y)$ is also a weak equivalence. From this, we easily deduce the following:

    \begin{lemma} \label{lem: induced on nerves}
        If $f: C \to D$ is a weak equivalence of small simplicial categories, then $Nf: NC \to ND$ is a weak equivalence of simplicial sets. 
    \end{lemma}
    % \begin{proof}
    %     First of all, fix an arbitrary $i \geq 0$ and consider the simplicial categories $\tilde{C}$ and $\tilde{D}$ obtained from $C$ and $D$ by keeping the same objects, but retaining only those simplices in each hom-set that may be expressed as a composition of length $i$. Evidently $C$ and $\tilde{C}$ are homotopy equivalent (as are $D$ and $\tilde{D}$): through either sufficiently many compositions with the identity map, or through sufficiently many concatenations, whichever is needed, every $n$-simplex in every hom-set can be expressed as a sequence of length $i$. In other words, have a weak equivalence $C_{i, *} \to D_{i, *}$. The result follows.
    % \end{proof}
    
    From Theorem \ref{main prop} and Lemma \ref{lem: induced on nerves}, we put together the following result:
    \begin{corollary} \label{cor: simp symm same K}
       Let $(S_1, \square_1)$ and $(S_2, \square_2)$ be good simplicial symmetric monoidal categories. If $F: (S_1, \square_1) \to (S_2, \square_2)$ is a weak equivalence, then the induced map between the nerves of their group completions, $NF: NS_1^{-1}S_1 \to NS_2^{-1}S_2$, is a weak equivalence of simplicial sets.
    \end{corollary}
    
    This corollary is particularly relevant in the context of the following definition for the K-theory for simplicial symmetric monoidal categories, leading directly to the main pursuit of the paper. 
    
    \begin{definition} \label{k theory def}
        Let $(S, \square)$ be a good simplicial symmetric monoidal category. The K-groups of $S$ are the homotopy groups of the geometric realisation of the nerve of $S^{-1}S$, that is,
        \[ K_n(S) = \pi_n(\lvert NS^{-1}S \rvert), \quad n \geq 0.
         \] We call the topological space $\lvert NS^{-1}S \rvert$ the \textit{K-theory space} of $S$, and denote it by $K(S)$. 
     \end{definition}
    
    \begin{proposition} \label{prop: funct 2}
        There is a functor $K: \mathrm{SymMonGpd^\dagger_{\Delta}} \to \mathrm{Top}$ given by $S \mapsto K(S)$ on objects, and $\phi \mapsto |N(\Lambda(\phi))|$ on morphisms (see Proposition \ref{prop: Lambda}).
    \end{proposition}
    
    \begin{proof}
        The map above is a composition of the three maps: 
        
        \[ \mathrm{SymMonGpd^\dagger_{\Delta}} \xrightarrow{\Lambda} \mathrm{SymMonCat_{\Delta}} \xrightarrow{N} \mathrm{sSet} \xrightarrow{|\,\cdot\,|} \mathrm{Top} \]
        
        where
        \begin{itemize}
            \item $\Lambda$ is as defined in Proposition \ref{prop: Lambda},
            \item $N$ is the nerve functor, and 
            \item $|\,\cdot\,|$ is the geometric realisation functor.\qedhere
        \end{itemize} 
    \end{proof}
    
     %Much could be said on why this is a useful, fruitful definition of K-theory. {\color{blue} This  part could be redone after the introduction and background are written.} A starting point for exploring its wide scope of applications is in the algebraic K-theory literature, which treats discrete symmetric monoidal categories; also 
    The definition we give above for the K-theory of a simplicial (or topological) symmetric monoidal category, which naturally generalises Quillen's original construction for discrete symmetric monoidal categories \cite{grayson2006higher}, is both sound and well founded: the space $|NS^{-1}S|$ is an infinite loop space, and infinite loop space machines accept topological or simplicial inputs \cite[pp.~331, 338]{weibel2013k}. The topologically enriched $S^{-1}S$ construction is explored from other angles in M.~Paluch's work \cite{paluch1991algebraic,paluch1996topology}.  
    
    Corollary \ref{cor: simp symm same K} now implies that a weak equivalence between good simplicial symmetric monoidal categories induces an isomorphism of topological K-theories, per Definition \ref{k theory def}. This is a pivotal step in achieving our main goal; we state it expressly as a theorem:
    
    \begin{theorem} \label{thm: K-theory equiv} A weak equivalence $F: (S_1, \square_1) \to (S_2, \square_2)$ of good simplicial symmetric monoidal categories induces an isomorphism of K-groups, $K_n(S_1) \to K_n(S_2)$, for all $n \geq 0$. 
    \end{theorem}
    
    \begin{remark}
        For topological symmetric monoidal groupoids where translations are faithful, we define the higher topological K-theory by first applying the singular functor to the morphism spaces. This converts the topological groupoid into a good simplicial symmetric monoidal category, to which the above results apply.
    \end{remark}
    
    In the following section, we define the simplicial, symmetric monoidal categories of holomorphic and topological $\alpha$-twisted vector bundles on a reduced Stein space, prove that they are weakly equivalent, and apply Theorem \ref{thm: K-theory equiv} to this weak equivalence, thereby establishing the main result of the paper.

\section{The Oka principle in higher twisted K-theory}

Twisted K-theory, originally introduced in the work of Donovan and Karoubi \cite{donovan1970graded}, later reformulated geometrically by Atiyah and Segal \cite{atiyah2004twisted, atiyah2006twisted}, is a sophisticated variant of classical K-theory. For each cohomology class $\alpha \in H^3(X,\mathbb{Z})$ of a topological space $X$, there is an associated $\alpha$-twisted K-group $K^0_\alpha(X)$, which admits several distinct but non-trivially equivalent definitions (cf. \cite{atiyah2004twisted, atiyah2006twisted, bouwknegt2002twisted}). We focus on the case where $\alpha$ is torsion, and adopt a modified version of the definition of $K^0_\alpha(X)$ in \cite{karoubi2012twisted} due to Karoubi, which uses $\alpha$-twisted vector bundles. Using the simplicially enriched $S^{-1}S$ construction studied above, we generalise the definition of $K^0_\alpha(X)$ to higher $\alpha$-twisted K-groups $K^{-n}_\alpha(X)$. While this definition of higher topological $\alpha$-twisted K-theory is both natural and consistent with existing formulations (cf. Section~\ref{sec 5}), it does not appear to have received any prior treatment in the literature. Additionally, we introduce the first definition of $\alpha$-twisted higher holomorphic K-theory, extending beyond the $K^0$ case mentioned in \cite{mathai2003chern}. 

The primary question of this work is the following: is the higher topological $\alpha$-twisted K-theory of a Stein space isomorphic to higher holomorphic $\alpha$-twisted K-theory? We prove in this section that the answer to this question is affirmative: in short, the Oka principle holds in higher, twisted K-theory.  Here is the precise statement:

\begin{centralthm}\label{thm:central}  
    Let $X$ be a reduced Stein space and $\alpha$ be a torsion class in $H^2(X, \cO^*) \simeq H^3(X, \ZZ)$. For all $n \geq 0$, the homomorphism of K-groups induced by inclusion, \[K_\alpha^{-n, \mathrm{\mathcal{O}}}(X) \to K_\alpha^{-n, \mathrm{\mathcal{C}}}(X),\] is an isomorphism. 
    \end{centralthm}

In the above, the terms $K_\alpha^{-n, \mathrm{\mathcal{O}}}(X)$ and $K_\alpha^{-n, \mathrm{\mathcal{C}}}(X)$ denote the K-groups associated with the symmetric monoidal categories of holomorphic and topological $\alpha$-twisted vector bundles respectively. A substantial portion of this section will be devoted to defining these categories, which become inputs for the simplicially enriched version of Quillen's $S^{-1}S$ functor (see Definition \ref{k theory def}).  The main idea of the proof is showing the statement to be a special case of the Theorem \ref{thm: K-theory equiv}, which we achieve by adapting the methods of Cartan and Grauert, and using the parametric Oka principle for generalised principal bundles (Theorem \ref{parametric Cartan}). 

%The version of higher topological $\alpha$-twisted K-theory we present in this chapter may be shown consistent with existing formulations (see Appendix \ref{app:the_second}). It is a natural notion; however, with no previous explicit appearance in the literature. This chapter also introduces the first definition of $\alpha$-twisted higher holomorphic K-theory, extending beyond the $K^0$ case mentioned in \cite{mathai2003chern}. Finally, in Section~5.3, we prove the Central Theorem, the Oka principle in higher twisted K-theory. The main idea of the proof is showing the statement to be a special case of the Theorem \ref{thm: K-theory equiv}, which we achieve by adapting the methods of Cartan and Grauert (cf. Chapter 3), and using the parametric Oka principle for generalised principal bundles (Theorem \ref{parametric Cartan}).

\subsection{The category of $\alpha$-twisted vector bundles}

Let $X$ be a Hausdorff, paracompact topological space. 
%t is well known that complex, finite-dimensional (ordinary) vector bundles are represented by their transition functions with respect to an open cover $\cU$ of $X$, which represent cohomology classes in $H^1(\cU, GL_n \cC)$, where $GL_n \mathcal{C}$ is the sheaf of invertible $n \times n$ matrices with continuous complex entries. %Twisted vector bundles, analogously, may be understood through the \textit{twisted} cohomology $H_\alpha^1(\cU, GL_n\cC)$; we begin this section by explaining this more precisely. 
Consider a sheaf $\cG$ of topological groups on $X$; for example, $\cG$ might be $GL_n\cC$, the sheaf of invertible $n \times n$ matrices with continuous complex entries. Given an open cover $\mathcal{U} = (U_i)_{i \in I}$ of $X$, we consider a 2-cocycle $(\alpha_{ijk})_{i,j,k \in I}$ taking values in $\cZ(\cG)$, the centre of the sheaf $\cG$. We assume this cocycle is \textit{completely normalised}; in full: 
\begin{enumerate}
    \item $\alpha_{ijk} = e_{\cG}$ (the identity element) whenever one of the indices $i,j,k$ is repeated,
    \item for any permutation $\sigma$ of indices $i, j, k \in I$, $\alpha_{\sigma(i)\sigma(j)\sigma(k)}  = \alpha_{ijk}^{\epsilon(\sigma)}$, where $\epsilon(\sigma)$ is the sign of $\sigma$, and 
    \item $\alpha_{ijl} \cdot \alpha_{jkl} = \alpha_{ikl} \cdot \alpha_{ijk}$ for all $i, j, k, l \in I$.
\end{enumerate}

One can show that every cocycle is cohomologous to a completely normalised one (see, for example, \cite{karoubi1997resolutions}). Working with completely normalised cocycles simplifies our subsequent computations and proofs. We note that in the case where $\cG = GL_n\cC$, $\alpha$ is a completely normalised 2-cocycle with coefficients in $\cC^*$, the sheaf of nonvanishing, continuous, complex-valued functions.

We now proceed to define the topological category of $\alpha$-twisted vector bundles over an open cover $\cU = (U_i)_{i \in I}$. 
%This definition is inspired by C\u{a}ld\u{a}raru's definition of twisted sheaves, as detailed in \tocite. While his definition serves well for his purposes (in the study of twisted sheaves of modules), it does not readily translate in an explicit, computable manner to the specifics of twisted vector bundles. The definition we introduce below aims to be as concrete and fundamental as possible -- to facilitate computations and proofs involving twisted vector bundles -- while minimising the assumptions required. {\color{blue} I have commented out all the discussion on Karoubi's version of twisted vector bundles, but I feel like there's probably still a place for it, more condensed perhaps, somewhere, as some people may wonder why I didn't just go with a definition that's something seemingly simpler?}
 To first put it in the most geometric way possible, over each open set $U_i$, the objects in this category are genuine (untwisted) $GL_n\CC$-principal bundles. These bundles are glued together on the intersections $U_i \cap U_j$ according to an $\alpha$-twisted cocycle condition. Morphisms between objects are defined as sections of the corresponding isomorphism bundle, as defined below. In particular, each morphism in this category is an isomorphism. A key property of this construction is that the category remains equivalent when defined using any refinement $\cV$ of the cover $\cU$ (where $\alpha$ is appropriately restricted). Here is a precise definition. 
\begin{definition} \label{def: cat twisted}Let $X$ be a Hausdorff, paracompact topological space, with an open cover $\cU = (U_i)_{i \in I}$, and let $\alpha$ be a completely normalised 2-cocycle of $\cC^*$ with respect to $\cU$. (As we later show in Proposition \ref{prop: cohom alpha}, this construction depends only on the cohomology class of $\alpha$ in $H^2(X, \cC^*) \simeq H^3(X,\ZZ)$.) We define the topologically enriched category $i\mathrm{Vect}^n_\alpha(\cU, \mathcal{C})$ of topological $\alpha$-twisted vector bundles as follows. An object $E$ -- which we call an $\alpha$-\textit{twisted vector bundle} -- is given by the following data: an open cover $(W_{a^i})_{{a^i} \in I^{(i)}}$ of $U_i$ for each $i \in I$, and a 1-cochain $\left(g_{a^{i}b^{j}}\right)_{a^i\in I^{(i)}, b^j\in I^{(j)}, i,j \in I}$ of the sheaf $GL_n\cC$, with respect to the open cover $\coprod_{i \in I}(W_{a^i})_{a^i \in I^{(i)}}$ of $X$, such that 
\begin{enumerate}
    \item on each $U_i$, $\left(g_{a^{i}b^{i}}\right)_{{a^i},{b^i} \in I^{(i)}}$ is a 1-cocycle with respect to $(W_{a^i})_{a \in I^{(i)}}$, and
    \item on the overlaps $W_{a^i} \cap W_{b^j} \cap W_{c^k}$, the $\alpha$-twisted cocycle condition $g_{a^{i}b^{j}}g_{b^{j}c^{k}} = g_{a^{i}c^{k}} \alpha_{ijk}$ holds.
\end{enumerate}
 We call $(g_{a^{i}b^{j}})_{a^i\in I^{(i)}, b^j\in I^{(j)}, i,j \in I}$ (which, from now on, we write as $(g_{a^{i}b^{j}})$ for visual clarity) an $\alpha$-twisted 1-cocycle with respect to $\cU$. The morphisms in this category are as follows: if objects $E$ and $F$ are defined by $\alpha$-twisted 1-cocycles $\left(g_{a^{i}b^{j}}\right)$ and $\left(f_{a^{i}b^{j}}\right)$ respectively, with respect to the open cover $\coprod_{i \in I}(W_{a^i})_{a^i \in I^{(i)}}$, a morphism from $E$ to $F$ is given by a 0-cochain $(c_{a^{i}})_{a \in I^{(i)}, i\in I}$ of $GL_n\cC$ with respect to the same cover, such that \[c_{a^{i}} = f_{a^{i}b^{j}}c_{b^{j}}g_{b^{j}a^{i}}.\] Note that every morphism is an isomorphism (the inverse of $(c_{a^{i}})$ is, of course, $(c^{-1}_{a^{i}})$).  The category $i\mathrm{Vect}^n_\alpha(\cU, \mathcal{C})$ is enriched in topological spaces, with the morphism space from $E$ to $F$ given by the space $\Gamma(\Iso(E, F))$ of sections -- equipped with the compact-open topology -- of the \textit{isomorphism bundle} $\Iso(E, F)$ from $E$ to $F$. Expressly, the generalised principal bundle $\Iso(E, F)$ can be defined by the trivialising cover $\coprod_{i \in I}(W_{a^i})_{a^i \in I^{(i)}}$ and transition functions $\psi_{a^ib^j}$ given by $h \mapsto f_{a^{i}b^{j}}(h)\cdot h \cdot g_{b^{j}a^{i}}(h)$,  which are set-theoretic automorphisms of $GL_n\CC$. The structure group bundle of $\Iso(E,F)$ is the \textit{automorphism bundle} $\Aut(E) = \Iso(E, E)$ of $E$; this is a $GL_n\CC$-group bundle which acts fibrewise on $\Iso(E, F)$ by right multiplication. Using the fact that $\alpha$ is completely normalised and takes values in $\CC^*$, we may easily show that $\Iso(E, F)$ and $\Aut(E)$ are genuine -- that is, untwisted -- bundles. In other words, if $E$ and $F$ are defined by $\alpha$-twisted cocycles $\left(g_{a^{i}b^{j}}\right)$ and $\left(f_{a^{i}b^{j}}\right)$ respectively; then, 
 \begin{equation*}
    \psi_{a^{i}b^{j}} \, \psi_{b^{j}c^{k}}  = \psi_{a^{i}c^{k}}.
\end{equation*}
 %Although the presentation in Section \ref{sec: homotopy between sections} is of isomorphism and automorphism bundles in the context of untwisted generalised principal bundles -- seemingly relevant only to untwisted vector bundles of finite rank $n$ -- the construction applies identically here. Indeed, it is evident that the isomorphism and automorphism bundles are defined locally in the same manner as in Section \ref{sec: homotopy between sections}, as twists are not a local property. Globally, the effects of the twists cancel. We note briefly how this is so for $\Iso(E,F)$ (the argument for $\Aut(E)$ is analogous). Say that $E$ and $F$ are defined by $\alpha$-twisted cocycles $\left(g_{a^{i}b^{j}}\right)$ and $\left(f_{a^{i}b^{j}}\right)$ respectively.
\end{definition}

The next stage involves constructing a family of equivalent categories that encompasses all topological $\alpha$-twisted vector bundles on $X$, uniting various characteristics: varying bundle ranks, different choices of covers  and refining maps, and distinct cocycles representing the same cohomology class as $\alpha$. We make a few remarks before proceeding.

\begin{remark} We consider $\left(g_{a^{i}b^{j}}\right)$ as a 1-cochain with respect to the \textit{disjoint} union of the open covers $(W_{a^i})_{a^i \in I^{(i)}}$, as $i$ runs over $I$. In other words, while it may be that the sets $W_{b^j}$ and $W_{a^i}$ are equal for $i \neq j$, where $b^j \in I^{(j)}$ and $a^i \in I^{(i)}$, we consider these open sets as distinct. 
\end{remark}
\begin{remark} Observe that the twisted cocycle condition reduces to the ordinary cocycle condition when $i = j = k$; so condition 1 of Definition \ref{def: cat twisted} is implied by condition 2. In geometric terms, over each $U_i$, the twisted vector bundle $E$ corresponds to an ordinary, untwisted vector bundle (equivalently, an untwisted principal bundle) -- but not necessarily a trivial bundle. 
    %This distinction becomes key when we prove (in Proposition \ref{prop: refine iso}) that repeating the above construction using a finer cover gives an equivalent topological category. We included both conditions in the definition to explicitly highlight that the objects are untwisted bundles on each open set $U_i$.
\end{remark}
\begin{remark} The category we have defined above is a  concrete adaptation of C\u{a}ld\u{a}raru's definition of twisted sheaves of modules \cite{caldararu2000derived} to twisted vector bundles. %which, while suited to the study of twisted sheaves of modules, does not readily translate in an explicit manner to the 
     %The definition above aims to be as concrete and fundamental as possible -- to facilitate computations and proofs involving twisted vector bundles -- while minimising the assumptions required. 
\end{remark}
\begin{remark}    \label{rem: comp with Karoubi}
We must briefly address why we did not opt for a simpler definition of the category of $\alpha$-twisted vector bundles -- as the full subcategory of the category above, in which the objects are trivial vector bundles on each open set $U_i$, which satisfy the $\alpha$-twisted cocycle condition on overlaps. In fact, this happens to be one of the first definitions of $\alpha$-twisted vector bundles to appear in the literature \cite{karoubi2012twisted}. Any definition must account for $\alpha$-twisted vector bundles that can be constructed in the same way on a different open cover $\cU'$, for which we need to consider refinements of covers (cf. Proposition \ref{prop: refine iso} below). Karoubi addresses this issue by appealing to the existence of good open covers and their cofinality. However, translating this approach to the holomorphic setting -- as we wish to do -- presents challenges; of note, we would need to understand the precise meaning of Stein good open covers in this context, and then consider the difficult question of the existence of (sufficiently many) of them. To avoid this, we work with the more general, and somewhat technical, definition above. 
\end{remark}

%Let us further examine the morphisms in the category $i\mathrm{Vect}^n_\alpha(\cU, \mathcal{C})$. Given two $\alpha$-twisted bundles $E$ and $F$, we can construct a generalised principal bundle $\Iso(E, F)$ of isomorphisms from $E$ to $F$. Its structure group bundle is $\Aut(E)$, the bundle of automorphisms of $E$, which is a group bundle over $X$ with fibre $GL_n\CC$. The detailed constructions and discussions on isomorphism and automorphism bundles can be found in Section \ref{sec: homotopy between sections}. Although the presentation in Section \ref{sec: homotopy between sections} is of isomorphism and automorphism bundles in the context of untwisted generalised principal bundles -- seemingly relevant only to untwisted vector bundles of finite rank $n$ -- the construction applies identically here. Indeed, it is evident that the isomorphism and automorphism bundles are defined locally in the same manner as in Section \ref{sec: homotopy between sections}, as twists are not a local property. Globally, the effects of the twists cancel. We note briefly how this is so for $\Iso(E,F)$ (the argument for $\Aut(E)$ is analogous). Say that $E$ and $F$ are defined by $\alpha$-twisted cocycles $\left(g_{a^{i}b^{j}}\right)$ and $\left(f_{a^{i}b^{j}}\right)$ respectively.

\begin{proposition} \label{prop: refine iso}
    If $\cV$ is a refinement of $\cU$, then the refinement functor $i\mathrm{Vect}^n_\alpha(\cU, \mathcal{C}) \to i\mathrm{Vect}^n_\alpha(\cV, \mathcal{C})$ is an equivalence of topological groupoids. 
 
\end{proposition}
\begin{proof}
Consider the functor $i\mathrm{Vect}^n_\alpha(\cU, \mathcal{C}) \to  i\mathrm{Vect}^n_\alpha(\cV, \mathcal{C})$, induced by refinement. That is, let $\cV = (V_s)_{s \in J}$ and $\cU = (U_i)_{i \in I}$, and fix a refinement map $\tau: J \to I$. If $E$ is an $\alpha$-twisted vector bundle in $i\mathrm{Vect}^n_\alpha(\cU, \mathcal{C})$, defined on each open set $U_i$ by untwisted bundles, say, $E_i$, the refinement functor maps $E$ to the $\alpha$-twisted vector bundle constructed from the untwisted bundles $F_s = E_{\tau(s)}|_{V_s}$ on $V_s$. That this bundle is $\alpha$-twisted on $\cV$ is evident, and we obtain the defining 1-cochain by pulling back along $\tau$ -- that is, if $E$ is defined by the 1-cochain $(g_{a^{i}b^{j}})$ with respect to the disjoint union of the open covers $(W_{a^i})_{a^i \in I^{(i)}}$ of $U_i$, as $i$ runs over $I$, then $F$ is given by the 1-cochain $(g_{a^{\tau(s)}b^{\tau(t)}})$ with respect to the disjoint union of the open covers $(W_{a^{\tau(s)}}\vert_{V_s})_{a^{\tau(s)} \in I^{\tau(s)}}$ of $V_s$, as $s$ runs over $J$. Morphisms are likewise mapped to their pullbacks along $\tau$. 

We will show that this functor is full, faithful and essentially surjective, and verify that the bijection between the morphism spaces is a homeomorphism. First, to see that it is full, consider objects $E_\cU$ and $F_\cU$ in the category $i\mathrm{Vect}^n_\alpha(\cU, \mathcal{C})$, and denote by $E_\cV$ and $F_\cV$ their respective images under the refinement functor. Assume that $E_\cU$ is defined by a 1-cochain $(g_{a^{i}b^{j}})$, and $F_{\cU}$ is defined by $(f_{a^{i}b^{j}})$, with respect to an open cover $\coprod_{i \in I}(W_{a^i})_{a^i \in I^{(i)}}$ of $X$, as above. Let $(\varphi_{a^{s}})$ define a section of the isomorphism bundle from $E_{\cV}$ to $F_{\cV}$; which means that $(\varphi_{a^{s}})$ is a 0-cochain of $GL_n\cC$ along $\coprod_{s \in J}(W_{a^{\tau(s)}}\vert_{V_s})_{a^{\tau(s)} \in I^{\tau(s)}}$, satisfying 
\[ \varphi_{a^{s}} = f_{a^{\tau(s)}b^{\tau(t)}}\, \varphi_{b^{t}} \, g_{b^{\tau(t)}a^{\tau(s)}}
    \] on $W_{a^{\tau(s)}}\vert_{V_s} \cap W_{b^{\tau(t)}}\vert_{V_t}$. 
We must find a section $(\psi_{c^{i}})$ of $\Iso(E_{\cU}, F_{\cU})$ which maps to $(\varphi_{a^{s}})$ under the refinement functor. Since 
\begin{align*}
    {\varphi_{a^{s}}}^{-1} f_{a^{\tau(s)}c^{i}}f_{c^{i}b^{\tau(t)}} \, \varphi_{b^{t}} 
    &= {\varphi_{a^{s}}}^{-1} f_{a^{\tau(s)}b^{\tau(t)}}\, \varphi_{b^{t}}\, \alpha_{\tau(s)i\tau(t)} \\ 
    &= g_{a^{\tau(s)}b^{\tau(t)}}\,\alpha_{\tau(s)i\tau(t)} \\
    &= g_{a^{\tau(s)}c^{i}}g_{c^{i}b^{\tau(t)}}
\end{align*}
on $W_{a^{\tau(s)}}\vert_{V_s} \cap W_{b^{\tau(t)}}\vert_{V_t} \cap W_{c^{i}}$, we find that $$f_{c^{i}a^{\tau(s)}} {\varphi_{a^{s}} g_{a^{\tau(s)}c^{i}}} = f_{c^{i}b^{\tau(t)}} {\varphi_{b^{t}} g_{b^{\tau(t)}c^{i}}},$$ which in turn defines a section of the sheaf $GL_n\cC$ on $W_{c^{i}}$, and hence a 0-cochain $(\psi_{c^{i}})$ on $\coprod_{i \in I}(W_{c^{i}})_{c^i \in I^{(i)}}$. Clearly, the restriction of $\psi_{a^{\tau(s)}}$ to $W_{a^{\tau(s)}}\vert_{V_s}$ is $\varphi_{a^{s}}$. Observe also that 
\begin{align*}
    f_{c^{i}d^{j}}\,\psi_{d^{j}}\,g_{d^{j}c^{i}}  %&= f_{c^{i}d^{j}} \, f_{d^{j}a^{\tau(s)}}\, {\varphi_{a^{s}} \,g_{a^{\tau(s)}d^{j}}} \,g_{d^{j}c^{i}} \\
    &= f_{c^ia^{\tau(s)}}\, {\varphi_{a^{s}} \,g_{a^{\tau(s)}c^{i}}} \, \alpha_{i j \tau(s)}\,\alpha_{\tau(s)j i} \\
   % &= f_{c^ia^{\tau(s)}}\, {\varphi_{a^{s}} \,g_{a^{\tau(s)}c^{i}}} \\
    &= \psi_{c^{i}}. 
    \end{align*}
So, $(\psi_{c^{i}})$ defines a section of $\Iso(E_\cU, F_\cU)$ -- this proves that the refinement functor is full.

Now, we show faithfulness. Assume we have two 0-cochains $(\psi_{c^{i}})$ and $(\tilde{\psi}_{c^{i}})$ along $\coprod_{i \in I}(W_{c^{i}})_{c^i \in I^{(i)}}$ defining sections of $\Iso(E_\cU, F_\cU)$, with identical images under the refinement functor, given by a 0-cochain $(\varphi_{a^{s}})$ with respect to $\coprod_{s \in J}(W_{a^{\tau(s)}}\vert_{V_s})_{a^{\tau(s)} \in I^{\tau(s)}}$. Take an open set $U_i$ in $\cU$; 
%Since $\cV = (V_s)_{s \in J}$ covers $X$, it covers $U_i$; so, the subcover $\left(U_{\tau(s)}\right)_{s \in J}$ covers $U_i$. 
every point of $U_i$ has a neighbourhood $W_{a^{\tau(s)}} \cap V_s \cap U_i$ on which the values of $\psi_{c^{i}}$ and $\tilde{\psi}_{c^{i}}$ are both given by $f_{c^{i}a^{\tau(s)}}\varphi_{a^{s}}g_{a^{\tau(s)}c^{i}}$. In other words, they are locally identical everywhere, and therefore define the same section. 

Moreover, the bijection $\Gamma(\Iso(E_{\cU},F_{\cU})) \to \Gamma(\Iso(E_{\cV},F_{\cV}))$ is a homeomorphism: it suffices to verify this locally, where the map reduces to the identity on sufficiently small open sets.
%Take an arbitrary section of $\Iso(E_{\cU},F_{\cU})$, and consider a neighbourhood $[K, U]_{\cU}$ (where $K \subset X$ is compact, and $U \subset \Iso(E_{\cU},F_{\cU})$ is open) so small that $K$ is contained within some $W_{a^{\tau(s)}}|_{V_s}$, and $U$ is a subset of $W_{a^{\tau(s)}}|_{V_s} \times GL_n\CC$. The restriction of the bijection to $[K,U]_{\cU} \to [K,U]_{\cV}$ is clearly a homeomorphism, which maps an open set $[K', U']_{\cU} \subset [K, U]_{\cU}$ to $[K', U']_{\cV} \subset [K, U]_{\cV}$, and whose inverse performs the opposite. 

Finally, we need to know that the functor is essentially surjective. Consider an object $E_\cV$ in $i\mathrm{Vect}^n_\alpha(\cV, \mathcal{C})$ defined by a 1-cochain $(g_{a^{s}b^{t}})$ with respect to the disjoint union of open covers $(W_{a^{s}})_{a^s \in J^{(s)}}$ of $V_s$, as $s$ runs over $J$ (that is, $\coprod_{s \in J}(W_{a^{s}})_{a \in J^{(s)}}$). From this, we define a twisted vector bundle $E_\cU$ in $i\mathrm{Vect}^n_\alpha(\cU, \mathcal{C})$ as follows: for each $U_i$, consider the open cover $\coprod_{s \in J}(U_i \cap W_{a^{s}})_{a \in J^{(s)}}$. Denote by $U_i^{a,s}$ the set $U_i \cap W_{a^{s}}$, and define transition functions from $U_i^{a,s}$ to $U_j^{b,t}$ by \[\psi_{(b,t,j)(a,s,i)} = \alpha_{ji\tau(t)}\alpha^{-1}_{i \tau(t) \tau(s)}g_{b^{t}a^{s}}.\] It can be easily shown that on each $U_i$, these transition functions define a genuine vector bundle.
%\begin{align*}\psi_{(c,u,i)(b,t,i)}\,\psi_{(b,t,i)(a,s,i)} %&= \alpha^{-1}_{i \tau(u) \tau(t)}\,g_{c^{u}b^{t}} \,\alpha^{-1}_{i \tau(t) \tau(s)}\,g_{b^{t}a^{s}} \\ 
%&= \alpha^{-1}_{i \tau(u) \tau(t)}\,\alpha^{-1}_{i \tau(t) \tau(s)}\,g_{c^{u}b^{t}}\,g_{b^{t}a^{s}} \\
%&= \alpha^{-1}_{i \tau(u) \tau(t)}\,\alpha^{-1}_{i \tau(t) \tau(s)}\,\alpha_{\tau(u)\tau(t)\tau(s)}\,g_{c^{u}a^{s}} \\
%&= \alpha^{-1}_{\tau(s)i\tau(u)}\,g_{c^{u}a^{s}} \\ 
%&= \alpha^{-1}_{i\tau(u)\tau(s)}\,g_{c^{u}a^{s}} \\ 
%&= \psi_{(c,u,i)(a,s,i)}.
%\end{align*}
Through elementary algebraic steps (which we omit), it can be shown that these genuine vector bundles come together to form an $\alpha$-twisted bundle on $\cU$; that is,
\begin{align*}
    \psi_{(c,u,k)(b,t,j)}\,\psi_{(b,t,j)(a,s,i)} 
    %&= \alpha_{kj\tau(u)}\,\alpha^{-1}_{j \tau(u) \tau(t)}\,g_{c^{u}b^{t}}\,\alpha_{ji\tau(t)}\,\alpha^{-1}_{i \tau(t) \tau(s)}\,g_{b^{t}a^{s}}\\
    % &= \alpha_{kj\tau(u)}\,\alpha^{-1}_{j \tau(u) \tau(t)}\,\alpha_{ji\tau(t)}\,\alpha^{-1}_{i \tau(t) \tau(s)}\,g_{c^{u}b^{t}}\,g_{b^{t}a^{s}}\\
    % &= \alpha_{kj\tau(u)}\,\alpha^{-1}_{j \tau(u) \tau(t)}\,\alpha_{ji\tau(t)}\,\alpha^{-1}_{i \tau(t) \tau(s)}\,\alpha_{\tau(u)\tau(t)\tau(s)}\,g_{c^{u}a^{s}} \\ 
    % &= \alpha_{kj\tau(u)}\,\alpha^{-1}_{j \tau(u) \tau(t)}\,\alpha_{ji\tau(t)}\, \alpha^{-1}_{i \tau(u)\tau(s)} \, \alpha_{\tau(t)i\tau(u)}\,g_{c^{u}a^{s}} \\ 
    % &= \alpha_{kji}\,\alpha^{-1}_{\tau(u)ji} \, \alpha_{\tau(u)ki}\, \alpha^{-1}_{j \tau(u) \tau(t)}\,\alpha_{ji\tau(t)}\, \alpha^{-1}_{i \tau(u)\tau(s)} \, \alpha_{\tau(t)i\tau(u)}\,g_{c^{u}a^{s}}\\ 
    % &= \alpha_{kji}\,\alpha_{\tau(u)ki}\,\alpha^{-1}_{i \tau(u)\tau(s)}\,g_{c^{u}a^{s}}\\ 
    &=\alpha_{kji}\, \psi_{(c,u,k)(a,s,i)}.
\end{align*}
Evidently, the refinement functor $i\mathrm{Vect}^n_\alpha(\cU, \mathcal{C}) \to i\mathrm{Vect}^n_\alpha(\cV, \mathcal{C})$ maps $E_\cU$ to $E_\cV$ (in this case, $j = \tau(t)$ and $i = \tau(s)$), which concludes the proof. 
\end{proof}

\begin{proposition} \label{prop: cohom alpha}
    If $\alpha, \tilde{\alpha} \in Z^2(\cU, \cC^*)$ are completely normalised, cohomologous 2-cocycles, then $i\mathrm{Vect}^n_\alpha(\cU, \mathcal{C})$ and $i\mathrm{Vect}^n_{\tilde{\alpha}}(\cU, \mathcal{C})$ are isomorphic topological groupoids. 
\end{proposition}

\begin{proof}
 Let $\cU = (U_i)_{i \in I}$. The difference between the 2-cocycles $\alpha$ and $\tilde{\alpha}$ is a coboundary, which means that there exists a 1-cochain $(\eta_{ij})_{i,j \in I}$ with respect to $\cU$, taking values in $\CC^*$, with $\eta_{ij} = \eta^{-1}_{ji}$ and $\eta_{ii}= 0$, such that $\alpha_{ijk} = \tilde{\alpha}_{ijk}\,\eta_{jk}\,\eta^{-1}_{ik}\,\eta_{ij}$. Define a functor $i\mathrm{Vect}^n_\alpha(\cU, \mathcal{C}) \to i\mathrm{Vect}^n_{\tilde{\alpha}}(\cU, \mathcal{C})$ which
 \begin{enumerate}
    \item sends an $\alpha$-twisted bundle $E$ in $i\mathrm{Vect}^n_\alpha(\cU, \mathcal{C})$, defined by a 1-cochain $(g_{a^{i}b^{j}})$, to the $\tilde{\alpha}$-twisted bundle $F$ in $i\mathrm{Vect}^n_{\tilde{\alpha}}(\cU, \mathcal{C})$ defined by $(g_{a^{i}b^{j}}\,\eta_{ji})$, and
    \item sends a morphism in $i\mathrm{Vect}^n_\alpha(\cU, \mathcal{C})$, defined by a cocycle $(\varphi_{c^i})$, to the morphism in $i\mathrm{Vect}^n_{\tilde{\alpha}}(\cU, \mathcal{C})$ defined by the same cocycle $(\varphi_{c^i})$.
 \end{enumerate}
   It is a straightforward verification that this functor is well defined, and induces an equivalence of categories. 
%    It is easy to see that $(g_{a^{i}b^{j}}\,\eta_{ji})$ satisfies the $\tilde{\alpha}$-twisted cocycle condition:  
% $
%     g_{c^{k}b^{j}}\,\eta_{jk}\,g_{b^{j}a^{i}}\,\eta_{ij} = \alpha_{kji} \,\eta_{jk}\,\eta_{ij}\,g_{c^{k}a^{i}} 
%     = {\tilde{\alpha}_{ijk}}^{-1}\,\eta_{ik}\,g_{c^{k}a^{i}}.
% $
% (Everything commutes, as all of these cochains take values in $\CC^*$.) A morphism $(\varphi_{c^{i}})$ from $\alpha$-twisted vector bundles $E$ to $F$, defined by, say, $(g_{a^{i}b^{j}})$ and $(f_{a^{i}b^{j}})$ respectively, is a morphism from the image of $E$ to the image of $F$ under the functor in question: $\varphi_{a^{i}} = f_{a^{i}b^{j}} \, \varphi_{b^{j}} \, g_{b^{j}a^{i}} = f_{a^{i}b^{j}} \, \varphi_{a^{i}} \, g_{b^{j}a^{i}} \,\eta_{ij}\,\eta_{ij}^{-1}$. The induced maps between morphism spaces in this case is clearly a homeomorphism. That this functor respects composition and identity is immediate. Its inverse is the functor which maps an $\tilde{\alpha}$-twisted vector bundle $F$, defined by $(f_{a^{i}b^{j}}\,\eta_{ji})$,  to the $\alpha$-twisted vector bundle $E$ given by $(f_{a^{i}b^{j}}\,\eta_{ij})$, and maps each morphism to itself.
\end{proof}

A consequence of the above result is the following: 
\begin{corollary}
    If $\cV = (V_s)_{s \in S}$ is a refinement of $\cU = (U_i)_{i \in I}$, defined by different refining maps $\tau, \sigma: S \to T$, then the groupoids $i\mathrm{Vect}^n_{\tau^*\alpha}(\cV, \mathcal{C})$ and $i\mathrm{Vect}^n_{\sigma^*\alpha}(\cV, \mathcal{C})$ are isomorphic.
\end{corollary}
\begin{proof}
By Proposition \ref{prop: cohom alpha}, we only need to show that $\tau^*\alpha$ and $\sigma^*\alpha$ are cohomologous 2-cocycles. Expressly, we need a 1-cochain $(\eta_{rs})_{r,s \in J}$ of $\cC^*$ with respect to $\cV =(V_{s})_{s \in J}$, such that $\alpha_{\tau(r)\tau(s)\tau(t)} = \alpha_{\sigma(r)\sigma(s)\sigma(t)}\,\eta_{st}\,\eta_{tr}\,\eta_{rs}$. We set $\eta_{rs} = \alpha_{\sigma(r)\tau(r)\tau(s)} \, \alpha^{-1}_{\tau(s)\sigma(r)\sigma(s)}$. We see easily that $\eta_{rs} = \eta_{sr}^{-1}$ -- this is a statement of the 2-cocycle condition satisfied by $\alpha$; and $\eta_{rr} = 1$ as $\alpha$ is completely normalised. Basic algebraic manipulation, and applying the 2-cocycle relation for various indices at each step, shows that 
\begin{align*}
     \alpha_{\sigma(r)\sigma(s)\sigma(t)}\,\eta_{st}\,\eta_{tr}\,\eta_{rs} %&= \alpha_{\sigma(r)\sigma(s)\sigma(t)} \, \eta_{rs} \, \eta_{st} \, \alpha_{\sigma(t)\tau(t)\tau(r)}\, \alpha_{\sigma(r)\sigma(t)\tau(r)}\\ 
    % &= \alpha_{\sigma(r)\sigma(s)\sigma(t)} \, \eta_{rs} \, \eta_{st} \, \alpha_{\sigma(r)\tau(t)\tau(r)}\,\alpha_{\sigma(r)\sigma(t)\tau(t)} \\
    % &= \alpha_{\sigma(r)\sigma(s)\sigma(t)} \, \eta_{st} \, \alpha_{\sigma(r)\tau(r)\tau(s)} \, \alpha_{\sigma(s)\sigma(r)\tau(s)} \, \alpha_{\sigma(r)\tau(t)\tau(r)}\,\alpha_{\sigma(r)\sigma(t)\tau(t)} \\
    % &= \left( \alpha_{\sigma(r)\sigma(s)\sigma(t)} \, \alpha_{\sigma(s)\sigma(r)\tau(s)} \right)  \left( \alpha_{\sigma(r)\tau(r)\tau(s)}\, \alpha_{\sigma(r)\tau(t)\tau(r)} \right) \, \eta_{st} \,\alpha_{\sigma(r)\sigma(t)\tau(t)}\\ 
    % &= \alpha_{\tau(s)\sigma(s)\sigma(t)}\, \alpha^{-1}_{\tau(s)\sigma(r)\sigma(t)}\,\alpha_{\tau(r)\tau(s)\tau(t)} \,\alpha^{-1}_{\sigma(r)\tau(s)\tau(t)}\,\eta_{st} \, \alpha_{\sigma(r)\sigma(t)\tau(t)} \\ 
    % &= \alpha_{\tau(s)\sigma(s)\sigma(t)}\, \left( \alpha^{-1}_{\tau(s)\sigma(r)\sigma(t)} \,\alpha^{-1}_{\sigma(r)\tau(s)\tau(t)} \alpha_{\sigma(r)\sigma(t)\tau(t)} \right)\alpha_{\tau(r)\tau(s)\tau(t)} \,\eta_{st} \\ 
    % &= \alpha_{\tau(s)\sigma(s)\sigma(t)}\, \alpha_{\tau(s)\sigma(t)\tau(t)}\,\alpha_{\sigma(s)\tau(s)\tau(t)}\,\alpha_{\sigma(t)\sigma(s)\tau(t)}  \, \alpha_{\tau(r)\tau(s)\tau(t)}\\ 
    &= \alpha_{\tau(r)\tau(s)\tau(t)},
\end{align*}
which establishes the result. (The basic steps are omitted here for brevity.) \end{proof}

The results above tell us that, as long as the cohomology class of $\alpha$ can be represented on an open cover $\cU$, the specific choice of representing cocycle is not important; and in fact, it only matters what $\alpha$ is in the colimit $H^2(X, \mathcal{C}^*)$ (or equivalently, $H^3(X, \ZZ)$). We assumed at the outset that $\alpha$ is torsion -- but note that we have not yet used this assumption. We see below that unless $\alpha$ is torsion, the category is empty. 
\begin{lemma}
\textnormal{(\cite{caldararu2000derived,grothendieck1955general, karoubi2012twisted}.)} Consider the category $i\mathrm{Vect}^n_\alpha(\cU, \mathcal{C})$ constructed as above, with twisting cocycle $\alpha \in Z^2(\cU, \cC^*)$ assumed to be completely normalised. If the category contains at least one object, then $\alpha$ is torsion. 
\end{lemma}
\begin{proof}
Firstly, observe that an $\alpha$-twisted vector bundle on $X$ -- defined by a family of untwisted $GL_n\CC$-principal bundles (equivalently, a family of complex vector bundles of rank $n$) over an open cover of $X$, with an $\alpha$-twisted gluing condition on the intersections -- naturally defines a $PGL_n\CC$-principal bundle on $X$. Isomorphism classes of $PGL_n\CC$-principal bundles on $X$ are represented by elements of the first cohomology group $H^1(X, PGL_n\cC)$. Since $X$ is paracompact, and $\cC^*$ is the centre of $GL_n\cC$, the short exact sequence of sheaves on $X$,
$ 1 \to \cC^* \to GL_n\cC \to PGL_n\cC \to 1,
$ induces an exact sequence of cohomology groups (cf. \cite[Chapter V]{grothendieck1955general}):

\[
H^0(X, \cC^*) \to H^0(X, GL_n\cC) \to H^0(X, PGL_n\cC) \xrightarrow{\partial} H^1(X, \cC^*) \]
\[ \to H^1(X, GL_n\cC) \to H^1(X, PGL_n\cC) \xrightarrow{\partial} H^2(X, \cC^*).
\]

The coboundary map $ \partial: H^1(X, PGL_n\mathcal{C}) \to H^2(X, \mathcal{C}^*) $ maps a cohomology class in $ H^1(X, PGL_n\mathcal{C}) $, represented by a cocycle $ (\tilde{g}_{ij}) $, to the cohomology class of the 2-cocycle $(\varphi_{ijk})$, defined by $ \varphi_{ijk} = g_{ij}\, g_{jk}\, g_{ki}$. Here, $ (g_{ij}) \in C^1(X, GL_n\mathcal{C}) $ is a lifting of $ (\tilde{g}_{ij})$ that is chosen to be completely normalised. 

We are assuming that an $\alpha$-twisted vector bundle of rank $n$ exists. Such an object is represented in $ H^1(X, PGL_n\mathcal{C})$, and mapped under $\partial$ to the cohomology class of $\alpha$ in $H^2(X, \cC^*)$. Consider now the short exact sequence 
$1 \to \mu_n \to SL_n\cC \to PGL_n\cC \to 1$ --
obtained by restricting the sheaves of the previous sequence -- for which we obtain a similar exact sequence of cohomology groups, and a boundary map $ \tilde{\partial}: H^1(X, PGL_n\mathcal{C}) \to H^2(X, \mathcal{C}^*) $. This yields a commutative diagram of sheaf morphisms, \[\begin{tikzcd}
    H^1(X, PGL_n\mathcal{C}) \arrow[r, "\partial"] \arrow[d, equal] & H^2(X, \mu_n) \arrow[d, hook] \\
    H^1(X, PGL_n\mathcal{C}) \arrow[r,"\tilde{\partial}"] & H^2(X, \mathcal{C}^*)
    \end{tikzcd},\]
showing expressly that $\alpha$ in $H^2(X, \cC^*)$ is $n$-torsion. 
\end{proof}

\begin{remark} The proof shows that if $\alpha$ is $n$-torsion -- where $n$ is the minimal integer for which $\alpha$ is torsion -- then $\alpha$-twisted vector bundles can only have ranks that are multiples of $n$; namely, $n$, $2n$, $3n$ etc.
\end{remark}
%We pause to reflect on what we have found so far: once we specify an integer $n$, and an $n$-torsion cohomology class $\alpha$ in $H^2(X, \cC^*)$ (the `twisting class'), we obtain the same category of $\alpha$-twisted vector bundles of rank $n$ on $X$ up to equivalence, no matter the choice of representing cocycle of $\alpha$, or the choice of the open cover in the construction, so long as $\alpha$ can be represented on it. 

The construction is now nearly operational for our purposes: it remains only to introduce a symmetric monoidal product to tie together twisted bundles of varying rank. 
Firstly, choose any open cover $\cU = (U_i)_{i \in I}$ of $X$ on which $\alpha$ can be represented. Let \[i\mathrm{Vect}_\alpha(\cU, \mathcal{C}) = \coprod_{n {\geq 0}} i\mathrm{Vect}^n_\alpha(\cU, \mathcal{C}),\]
and define the product \[\square: i\mathrm{Vect}_\alpha(\cU, \mathcal{C}) \times i\mathrm{Vect}_\alpha(\cU, \mathcal{C}) \to i\mathrm{Vect}_\alpha(\cU, \mathcal{C}), \quad (E,F) \mapsto E \square F,\] where $E$ and $F$ are $\alpha$-twisted vector bundles on $X$ of ranks $n$ and $m$ respectively, and $E \square F$ is an $\alpha$-twisted bundle of rank $n+m$, given by the direct sum $E|_{U_i} \oplus F|_{U_i}$ on each open set $U_i$. It is readily seen that $(i\mathrm{Vect}_\alpha(\cU, \mathcal{C}), \square)$ is a symmetric monoidal category. 

%\begin{proof}
%It is straightforward to see that the monoidal product is well defined. The unit is the rank 0 vector bundle, and the associativity and commutativity constraints are defined fibrewise. We can also see that the symmetric monoidal structure, as it is locally given by the direct sum of vector bundles, is compatible with topological structures of the morphism spaces. 
%\end{proof}

%\begin{remark} It may appear that the definition of $(i\mathrm{Vect}_\alpha(\cU, \mathcal{C}), \square)$ above is only useful when $X$ is connected, and that we neglect $\alpha$-twisted vector bundles with varying ranks across different connected components of $X$; however, this is not materially so. If $X$ has several connected components (indexed by, say, $k \in K$), we refine $\cU$ as necessary, and consider a partition $(\cU^k)_{k \in K}$ of the open cover. In this case we slightly modify the definition of $i\mathrm{Vect}^n_\alpha(\cU, \mathcal{C})$ to be $\prod_{k \in K} i\mathrm{Vect}^n_\alpha(\cU^k, \mathcal{C})$ and the symmetric monoidal product to be $(E, F)\mapsto (E|_{X_k} \square F|_{X_k})_{k \in K}$. Hence it is fine to only work with twisted vector bundles of constant rank, which amounts to considering the connected components of $X$ one at a time. 
%\end{remark}

Observe that, up to this point of the construction, we have been careful to not employ any argument or technique that pertains only to topological objects and not to holomorphic ones (provided the base space $X$ admits holomorphic charts). Therefore, in an exactly analogous fashion, we may construct the symmetric monoidal category of \textit{holomorphic} $\alpha$-twisted vector bundles over a reduced complex space $X$, which we denote by $(i\mathrm{Vect}_\alpha(\cU, \mathcal{O}), \square)$. This involves replacing the sheaves $GL_n\cC$ and $\cC^*$ by $GL_n\cO$ and $\cO^*$ respectively, and following through all of the arguments above. 

We next address functoriality. The construction behaves naturally under maps between spaces, as shown by a routine proof (with details omitted for brevity). The following proposition is stated in the topological setting but holds analogously for the holomorphic case. 

\begin{proposition} \label{prop: funct 3}
    Let $X, Y,$ and  $Z$ be paracompact Hausdorff spaces, $\mathcal{U} = (U_i)_{i \in I}$ an open cover of $Y$, and $\alpha \in Z^2(\cU , \cC^*)$ a completely normalised 2-cocycle of $\cC^*$ with respect to $\cU$. Then,
    
    \begin{enumerate}
        \item for any continuous map $\phi: X \to Y$, there exists a well-defined topological pullback functor:
        $$\phi^*: i\mathrm{Vect}_\alpha(\mathcal{U}, \mathcal{C}) \to i\mathrm{Vect}_{\phi^*\alpha}(\phi^{-1}(\mathcal{U}), \mathcal{C}),$$
        
        \item for composable maps $\phi: X \to Y$ and $\psi: Z \to X$, 
        $(\phi \circ \psi)^* = \psi^* \circ \phi^*$, and
        
        \item if $\mathrm{id}_Y: Y \to Y$ is the identity map, then 
        $\mathrm{id}_Y^* = \mathrm{id}_{i\mathrm{Vect}_\alpha(\mathcal{U}, \mathcal{C})}$.
    \end{enumerate}
    \end{proposition}

\begin{proof}
   
The proofs of statements (2) and (3) use standard arguments, and are omitted.  For (1), we define the pullback functor below, but omit the routine verification of well-definedness.
   
    We denote by $\phi^*\alpha$ the pullback of $\alpha$ along $\phi$, which is a completely normalised $2$-cocycle with respect to the open cover $\phi^{-1}(\cU)$ of $X$. 
%\begin{enumerate}
   Let $\phi: X \to Y$ be a continuous map, and define \[\phi^*: i\mathrm{Vect}_\alpha(\mathcal{U}, \mathcal{C}) \to i\mathrm{Vect}_{\phi^*\alpha}(\phi^{-1}(\mathcal{U}), \mathcal{C})\] as follows.
    \begin{itemize}
        \item Say $E$ is an object of $i\mathrm{Vect}_\alpha(\mathcal{U}, \mathcal{C})$ given by the following data (cf.~Definition \ref{def: cat twisted}): an open cover $(W_{a^i})_{{a^i} \in I^{(i)}}$ of $U_i$ for each $i \in I$, and a 1-cochain $\left(g_{a^{i}b^{j}}\right)_{a^i\in I^{(i)}, b^j\in I^{(j)}, i,j \in I}$ of the sheaf $GL_n\cC$, with respect to $\coprod_{i \in I}(W_{a^i})_{a^i \in I^{(i)}}$. Then $\phi^*$ sends $E$ to the $\phi^*\alpha$-twisted vector bundle $\phi^*E$, given with respect to $\phi^{-1}(\cU)$, which is defined by: the open cover  $(\phi^{-1}(W_{a^i}))_{{a^i} \in I^{(i)}}$ of $\phi^{-1}(U_i)$ for each $i \in I$, and the 1-cochain $\left(g_{a^{i}b^{j} } \circ \phi\right)$ of $GL_n\cC$, with respect to $\coprod_{i \in I}(\phi^{-1}(W_{a^i}))_{a^i \in I^{(i)}}$. 
        \item A morphism in $i\mathrm{Vect}_\alpha(\mathcal{U}, \mathcal{C})$ from $E$ to $F$ -- objects which are defined by, say, $\alpha$-twisted cocycles $\left(g_{a^{i}b^{j}}\right)$ and $\left(f_{a^{i}b^{j}}\right)$ respectively, with respect to the open cover $\coprod_{i \in I}(W_{a^i})_{a^i \in I^{(i)}}$ -- that is given by a 0-cochain $(c_{a^{i}})_{a \in I^{(i)}, i\in I}$ of $GL_n\cC$ with respect to the same cover $\coprod_{i \in I}(W_{a^i})_{a^i \in I^{(i)}}$, is mapped under $\phi^*$ to the morphism $\phi^*(c_{a^{i}}): \phi^*E \to \phi^*F$. The morphism $\phi^*(c_{a^{i}})$ is defined as the 0-cochain $(c_{a^{i}} \circ \phi)$ with respect to the open cover $\coprod_{i \in I}(\phi^{-1}(W_{a^i}))_{a^i \in I^{(i)}}$ of $X$. 
    \end{itemize} %It is evident, by the way that they are defined, that the images under $\phi^*$ of the objects and the morphisms in $i\mathrm{Vect}_\alpha(\mathcal{U}, \mathcal{C})$ meet the requirements for objects and morphisms in $i\mathrm{Vect}_{\phi^*\alpha}(\phi^{-1}(\mathcal{U}), \mathcal{C})$, per Definition \ref{def: cat twisted}. 
    To verify well-definedness, one needs to show that map $\phi^*$ is a functor which is both monoidal and topological (that is, all induced maps on hom-sets are continuous): this is straightforward. %That $\phi^*$ preserves the identity morphism is evident by definition (if $c_{a^i}$ is of constant value $1_{n \times n}$, then so is $c_{a^i} \circ \phi$). Preservation of composition is also straightforward: $\phi^*(c_{a^i}d_{a^i})$ is equal to $(c_{a^i}d_{a^i})\circ \phi$ and $(c_{a^i} \circ \phi)(d_{a^i}\circ \phi)$, which is just matrix multiplication at each point. So, $\phi$ is a functor. As the symmetric monoidal structures of $i\mathrm{Vect}_\alpha(\mathcal{U}, \mathcal{C})$ and $i\mathrm{Vect}_{\phi^*\alpha}(\phi^{-1}(\mathcal{U}), \mathcal{C})$ are given locally by direct sum, they are compatible everywhere with the pullback functor. Finally, to see that $\phi^*$ is topological functor, take objects $E$ and $F$ in $i\mathrm{Vect}_\alpha(\mathcal{U}, \mathcal{C})$ and consider the induced map on morphism spaces, $\Gamma(\mathrm{Iso}(E, F)) \to \Gamma(\mathrm{Iso}(\phi^*E, \phi^*F))$; this map is defined by precomposition with $\phi$, and is continuous because $\phi$ is continuous. 
\end{proof} 

\begin{remark}
    We write $\phi^*$ whenever it is clear, but when we need to make a distinction between the functors $\phi^*: i\mathrm{Vect}_\alpha(\mathcal{U}, \mathcal{C}) \to i\mathrm{Vect}_{\phi^*\alpha}(\phi^{-1}(\mathcal{U}), \mathcal{C})$ and $\phi^*: i\mathrm{Vect}_\alpha(\mathcal{U}, \mathcal{O}) \to i\mathrm{Vect}_{\phi^*\alpha}(\phi^{-1}(\mathcal{U}), \mathcal{O})$, such as in the proof of the main theorem, we will use $\phi^*_\cC$ and $\phi^*_\cO$ respectively. 
\end{remark}

Recall that the ordinary K-theory $K^{\cO}(X)$ of a Stein space $X$ is the group completion of the abelian monoid of holomorphic vector bundles on $X$. Likewise, the topological K-theory $K^{\cC}(X)$ is obtained from the monoid of continuous vector bundles on $X$.  Below, we provide definitions for the higher holomorphic and topological $\alpha$-twisted K-theories of $X$. 

\subsection{Higher $\alpha$-twisted K-theory of $X$} \label{sec: twisted k theory for X}

As above, let $X$ be a reduced Stein space, and assume that $\alpha \in H_{\mathrm{tor}}^2(X, \mathcal{O}^*)$, or equivalently, in $H_{\mathrm{tor}}^2(X, \mathcal{C}^*)$. Let $\mathcal{U} = (U_i)_{i \in I}$ be an open cover of $X$ on which $\alpha$ can be represented as a completely normalised, holomorphic 2-cocycle. It is helpful to slightly abuse notation and denote both the representing cocycle and the cohomology class by $\alpha$ (Proposition \ref{prop: cohom alpha} justifies this choice). We build the definitions of higher holomorphic and topological $\alpha$-twisted $K$-theory as enriched versions of Quillen's $S^{-1}S$ construction.  As far as we are aware, this is the first definition of higher holomorphic twisted K-theory to be introduced to the literature.

%In the previous section, we constructed (up to equivalence) the symmetric monoidal, topological groupoids of $\alpha$-twisted holomorphic and topological vector bundles on $X$, namely $(i\mathrm{Vect}_\alpha(\cU, \cO), \square)$ and $(i\mathrm{Vect}_\alpha(\cU, \cC), \square)$ respectively. The ordinary holomorphic (respectively, topological) K-theory for a Stein space proceeds by taking the group completion of the abelian monoid of holomorphic (respectively, topological) vector bundles on $X$. We do an analogous operation here. 

Recall that in Section \ref{gp completion} we generalised the discrete $S^{-1}S$ construction to good simplicial symmetric monoidal categories $S$ (cf. Definition \ref{def: good}). This construction can be applied to the topological groupoids $(i\mathrm{Vect}_\alpha(\cU, \cO), \square)$ and $(i\mathrm{Vect}_\alpha(\cU, \cC), \square)$, after converting them to simplicial groupoids, and ensuring that they are good. First, we observe easily that the following result holds:

\begin{lemma}
If $E$ is an object of $(i\mathrm{Vect}_\alpha(\cU, \cO), \square)$ or $(i\mathrm{Vect}_\alpha(\cU, \cC), \square)$, then the functor $E\,  \square \, -$ is faithful. 
\end{lemma}
% \begin{proof}
% Assume $E$ is of rank $n$, and consider $\alpha$-twisted bundles $F$ and $G$ of the same rank $m$ (to ensure that the morphism space is not empty). The induced map on morphism spaces is 
% \[
% \Gamma(\mathrm{Iso}(F, G)) \to \Gamma(\mathrm{Iso}(E \, \square \, F, E \, \square \, G)), \quad \varphi \mapsto \mathrm{id} \square \varphi,
% \] where $\mathrm{id} \in \Gamma(\mathrm{Aut}(E))$ is the identity section. If $\mathrm{id} \square \varphi = \mathrm{id} \square \psi$ for $\varphi, \psi \in \Gamma(\mathrm{Iso}(F, G))$, then on a sufficiently small neighbourhood $V$ of any point $x \in X$, $I_{n} \oplus \varphi = I_{n} \oplus \psi$ where $I_{n} \in GL_n\CC$ is the identity matrix, which implies that $\varphi$ and $\psi$ are locally the same everywhere, and define the same section. 
% \end{proof}

By applying the singular functor to the morphism spaces, we convert $(i\mathrm{Vect}_\alpha(\cU, \cO), \square)$ and $(i\mathrm{Vect}_\alpha(\cU, \cC), \square)$ to simplicial symmetric monoidal groupoids. (For convenience, whenever clear, we use the same notation for both the simplicially enriched and topologically enriched categories of $\alpha$-twisted vector bundles. This abuse of notation is justified in our situation -- where we are primarily concerned with weak equivalences between hom-sets -- as the model categories of simplicial sets and topological spaces are Quillen equivalent.) We now define the holomorphic, $\alpha$-twisted K-groups (cf. Definition \ref{k theory def}) with respect to $\cU$ as
\[ K_\alpha^{-n, \cO}(\cU) = \pi_n\lvert N\, (i\mathrm{Vect}_\alpha(\cU, \cO))^{-1}i\mathrm{Vect}_\alpha(\cU, \cO)\rvert.
    \]
Similarly, define the topological, $\alpha$-twisted K-groups with respect to $\cU$ as
\[ K_\alpha^{-n, \cC}(\cU) = \pi_n\lvert N\,(i\mathrm{Vect}_\alpha(\cU, \cC))^{-1}i\mathrm{Vect}_\alpha(\cU, \cC)\rvert.
    \]
We call the topological spaces $K_\alpha^{\cO}(\cU) = \lvert N\, (i\mathrm{Vect}_\alpha(\cU, \cO))^{-1}i\mathrm{Vect}_\alpha(\cU, \cO) \rvert$ and $K_\alpha^{\cC}(\cU) = \lvert N\,(i\mathrm{Vect}_\alpha(\cU, \cC))^{-1}i\mathrm{Vect}_\alpha(\cU, \cC)\rvert$ the holomorphic and topological K-theory spaces with respect to $\cU$, respectively.

\begin{remark} We briefly remark on our choice to place $-n$ as a superscript, rather than $n$ as a subscript, as it appears in the previous section. The choice is not of particular importance, and is mainly to follow historic convention as best as we can. Algebraic K-theory uses subscripts as it often defines covariant functors. Topological K-theory, on the other hand, typically uses superscripts as it determines contravariant functors, and represents a generalised cohomology theory. 
    %We use subscript notation in Chapter 4 to emphasise the connection with the higher algebraic K-theory definition that inspires our definition. In this chapter, however, as we precompose the K-theory functor of Chapter 4 with the contravariant functor $X \mapsto i\mathrm{Vect}_\alpha(\cU, \cO)$ -- from the category of Hausdorff, paracompact topological categories to the category of simplicial symmetric monoidal categories --  it is more appropriate to align with the conventions of topological K-theory. 
    %Since topological K-theory normally satisfies the axioms of a generalised cohomology theory, w
    Since we precompose the K-theory functor of Section \ref{sec:homotopy} with the contravariant functor $X \mapsto i\mathrm{Vect}_\alpha(\cU, \cO)$, we add the negative sign in front of $n$ to ensure that the arrows in associated long exact sequences go in the correct direction.\end{remark}
 
We find that the K-theory functor above does not depend on the choice of open cover $\cU$ -- that is, if $\cV$ is a refinement of $\cU$ on which $\alpha$ can be represented, then the K-groups defined using $\cV$ are isomorphic to the ones above.

\begin{proposition}
    If $\cV$ is a refinement of $\cU$, then for all $n {\geq 0}$, $K_\alpha^{-n, \cC}(\cU) \simeq K_\alpha^{-n, \cC}(\cV)$ and $K_\alpha^{-n, \cO}(\cU) \simeq K_\alpha^{-n, \cO}(\cV)$. 
\end{proposition}
\begin{proof}
 We know that $i\mathrm{Vect}_\alpha(\cU, \cC)$ and $i\mathrm{Vect}_\alpha(\cV, \cC)$ (or $i\mathrm{Vect}_\alpha(\cU, \cO)$ and $i\mathrm{Vect}_\alpha(\cV, \cO)$) are equivalent as topological (or simplicial) groupoids. The K-theory functor above is a particular case of the functor in Definition \ref{k theory def}, which sends, by Theorem \ref{thm: K-theory equiv}, weakly equivalent simplicial symmetric monoidal categories to isomorphic K-groups.
\end{proof}

The above proposition allows us to rewrite $K_\alpha^{-n, \cC}(\cU)$ simply as $K_\alpha^{-n, \cC}(X)$, and $K_\alpha^{-n, \cO}(\cU)$ as $K_\alpha^{-n, \cO}(X)$; we call these groups holomorphic and topological $\alpha$-twisted K-groups of $X$, respectively.

\subsection{The higher $\alpha$-twisted Oka principle}
We have assembled almost all of the machinery needed to prove the main theorem, the Oka principle in higher, twisted K-theory. We need one more auxiliary lemma, after which we will restate and prove the theorem.

\begin{definition}
    Let $X$ be a Stein space, and let $\alpha \in H^2(X,\mathcal{O}^*)$ be a torsion class. Let $\mathcal{U} = (U_i)_{i \in I}$ be an open cover of $X$ on which $\alpha$ can be represented as a completely normalised 2-cocycle (which, slightly abusing notation, we denote throughout this section by $\alpha = (\alpha_{ijk})$). We say that an open set $V \subset X$ is \textit{good} if it satisfies the following condition:
    for every topological $\alpha$-twisted vector bundle $E$ on $V$ with respect to $\cU$, there exists a holomorphic $\alpha$-twisted vector bundle $F$ on $V$ with respect to $\cU$, such that $E$ and $F$ are isomorphic as topological $\alpha$-twisted vector bundles.
\end{definition}

%Practically, the above definition means the following. Say $E$ is a topological $\alpha$-twisted vector bundle on $V$, given with respect to $\cU|_V = (U_i|_V)_{i \in I}$, defined by: an open cover $(W_{a^i}|_V)_{{a^i} \in I^{(i)}}$ of $U_i|_V$ for each $i \in I$, and a 1-cochain $\left(g_{a^{i}b^{j}}\right)_{a^i\in I^{(i)}, b^j\in I^{(j)}, i,j \in I}$ of the sheaf $GL_n\cC$, with respect to the open cover $\coprod_{i \in I}(W_{a^i}|_V)_{a^i \in I^{(i)}}$ of $V$, satisfying the conditions of Definition \ref{def: cat twisted}. We continue to write these as $(g_{a^ib^j})$, $(W_{a^i})$ and $\coprod_{i \in I}(W_{a^i})$ whenever it is not ambiguous. The open set $V$ is said to be good if there exists a holomorphic $\alpha$-twisted bundle $F$, given by a 1-cochain $\left(f_{a^{i}b^{j}}\right)$ of the sheaf $GL_n\cO$, along with 0-cochain $(c_{a^{i}})_{a^i \in I^{(i)}, i\in I}$ of $GL_n\cC$, both with respect to the same cover $\coprod_{i \in I}(W_{a^i})$, such that \[c_{a^{i}} = f_{a^{i}b^{j}}c_{b^{j}}g_{b^{j}a^{i}}.\]

\begin{definition} A \textit{compact box} in $\CC^n$ is a set \[\{(x_1 + iy_1, x_2 + iy_2, \cdots, x_n + iy_n) \in \CC^n \, \vert \, a_i \leq x_i \leq b_i \textnormal{ and } a_i' \leq y_i \leq b_i'\}\] for real numbers $a_i \leq b_i$ and $a_i' \leq b_i'$. We refer to the interior of such a set as an \textit{open, relatively compact box} in $\CC^n$. 
\end{definition}

The following lemma is the $\alpha$-twisted analogue of \cite[Lemma in  \S5]{cartan1958espaces}. Our proof follows similar ideas, while incorporating geometric interpretations and making some arguments more explicit. 
\begin{lemma} \label{again compact box induction}
    Let $\Gamma \subset \CC^n$ be an open, relatively compact box of real dimension $k$, with one edge $J$ a union of two open intervals $J'$ and $J''$. Let $\cU = (U_i)_{i \in I}$ be an open cover of $\Gamma$, and say $E$ is a topological $\alpha$-twisted bundle on $\Gamma$, given by an open cover $(W_{a^i})_{{a^i} \in I^{(i)}}$ of $U_i$ for each $i \in I$, and a 1-cochain $\left(g_{a^{i}b^{j}}\right)$ of the sheaf $GL_n\cC$, with respect to the open cover $\coprod_{i \in I}(W_{a^i})_{a^i \in I^{(i)}}$ of $\Gamma$, satisfying the conditions of Definition \ref{def: cat twisted}. Assume without loss of generality that $\Gamma = J \times S_{k-1}$, where $S_{k-1} \subset \RR^{k-1}$. Let $\mathrm{pr}: \Gamma \to J$ be the canonical projection, and denote by $\Gamma'$ and $\Gamma''$ the subsets $\mathrm{pr}^{-1}(J')$ and $\mathrm{pr}^{-1}(J'')$ respectively. Assume $V \subset X$ is an open subset of a Stein space $X$, which is isomorphic to an analytic subspace of $\Gamma$. If the sets $V \cap \Gamma'$ and $V \cap \Gamma''$ are good, then $V$ is good. 
    \end{lemma}
 
    \begin{proof}
    Assume, as we may, that the open covers in question -- that is, $(U_i)$ and $(W_{a^i})$ -- comprise holomorphically convex subsets of $V$. The first step is to note that $\Gamma = \Gamma' \cup \Gamma''$, and the intersection $\Gamma' \cap \Gamma''$ is an open, relatively compact box. Let $V' = V \cap \Gamma'$ and $V'' = V\cap \Gamma''$. By assumption, $E$ is isomorphic to holomorphic $\alpha$-twisted vector bundles on $V'$ and $V''$. This means that there exist 0-cochains $({c_{a^i}}')_{a^i \in I^{(i)}, i \in I}$ and $({c_{a^i}}'')_{a^i \in I^{(i)}, i \in I}$ of $GL_n\cC$ with respect to  $\coprod_{i \in I}(W_{a^i}|_{V'})_{a^i \in I^{(i)}}$ and  $\coprod_{i \in I}(W_{a^i}|_{V''})_{a^i \in I^{(i)}}$ respectively, such that the $\alpha$-twisted vector bundles $F'$ and $F''$, defined by the 1-cochains $({f_{a^ib^j}}') = ({c_{a^i}}'^{-1}g_{a^ib^j}{c_{b^j}}')$ and $({f_{a^ib^j}}'') = ({c_{a^i}}''^{-1}g_{a^ib^j}{c_{b^j}}'')$, with respect to  $\coprod_{i \in I}(W_{a^i}|_{V'})_{a^i \in I^{(i)}}$ and  $\coprod_{i \in I}(W_{a^i}|_{V''})_{a^i \in I^{(i)}}$ respectively, are holomorphic -- that is, they take values in the sheaf $GL_n\cO$. This implies that on $V' \cap V''$, $F'$ and $F''$ are topologically isomorphic $\alpha$-twisted vector bundles. Indeed, on $U_{ij} \cap V' \cap V''$, we have 
    \[
        {f_{a^ib^j}}'' = {h_{a^i}}^{-1} {f_{a^ib^j}}' h_{b^j},
    \] 
    where $(h_{a^i})_{a^i \in I^{(i)}, i \in I}$ is given by $h_{a^i} = {c_{a^i}}'^{-1}{c_{a^i}}''$.
    
    From this, we deduce that $F'$ and $F''$ are also holomorphically isomorphic on $V' \cap V''$ -- that is, they are isomorphic in the category $(i\mathrm{Vect}_\alpha(\cU, \mathcal{O}), \square)$. Here is why: firstly, note that $V' \cap V''$, being an analytic subspace of an open cube, is Stein. Now, consider the isomorphism bundle $\mathrm{Iso}(F'', F')$ from $F''$ to $F'$. As we noted in Definition \ref{def: cat twisted}, $\mathrm{Iso}(F'', F')$ is an untwisted generalised principal bundle, and the 0-cochain $(h_{a^i})_{a^i \in I^{(i)}, i \in I}$ defines a continuous, global section of this bundle. The structure group bundle of $\mathrm{Iso}(F'', F')$ is the automorphism bundle $\Aut(F'')$ of $F''$. Since $V' \cap V''$ is Stein, and $\mathrm{Iso}(F'', F')$ is a genuine generalised principal bundle, Theorem \ref{Cartan's thm 1} tells us that there exists a holomorphic section of $\mathrm{Iso}(F'', F')$ that is homotopic to the continuous section $(h_{a^i})$. Explicitly, there is a family of continuous sections $(H_{a^i}^t)_{a^i \in I^{(i)}, i \in I}$ of $\mathrm{Iso}(F'', F')$, indexed by $t \in [0,1]$, such that $H_{a^i}^0 = h_{a^i}$, and $k_{a^i} = H_{a_i}^1$ takes values in $GL_n\cO$. In particular, for each $t \in [0,1]$, \[{f_{a^ib^j}}'' = {H^t_{a^i}}^{-1} {f_{a^ib^j}}' H^t_{b^j}.\]
   % , and

    Now consider, for arbitrary fixed $a^i \in I^{(i)}, i \in I$, the space $W_{a^i} \cap V$. The pair of open sets $W_{a^i} \cap V'$ and $W_{a^i} \cap V''$ form an open cover of this space, with respect to which $h_{a^i}$ and $k_{a^i}$ are 1-cocycles -- the cocycle condition is, of course, trivially satisfied. Refining the open cover $(W_{a^i})$ if need be, we may assume that $\Aut(F'')$ is trivial over $W_{a^i} \cap V$; then, from standard principal bundle theory, we deduce that the principal bundles on $W_{a^i} \cap V$ defined by the 1-cocycles $h_{a^i}$ and $k_{a^i}$ with respect to the cover $\{W_{a^i}\cap V', W_{a^i} \cap V''\}$ are topologically isomorphic. Further, note that $h_{a^i}$ splits on $W_{a^i} \cap V$ with respect the above cover -- indeed, $h_{a^i} = {c_{a^i}}'^{-1}{c_{a^i}}''$ --  which means that the corresponding principal bundle on $W_{a^i} \cap V$ is topologically trivial. Naturally, this implies that the principal bundle defined by $k_{a^i}$ is also topologically trivial. In fact, as $W_{a^i} \cap V$ is Stein, and this principal bundle is classical and untwisted, \cite[Theorem A]{cartan1958espaces} implies that this bundle also holomorphically trivial. 
    In other words, $k_{a^i}$ splits holomorphically with respect to $\{W_{a^i}\cap V', W_{a^i} \cap V''\}$: there exist holomorphic maps ${h_{a^i}}^\prime: W_{a^i}\cap V' \to GL_n\cO$ and ${h_{a^i}}'': W_{a^i}\cap V'' \to GL_n\cO$ such that $k_{a^i} = {h_{a^i}}^{\prime-1}{h_{a^i}}^{\prime\prime}$ on $W_{a^i} \cap V^\prime \cap V^{\prime\prime}$. 
    
    This implies that ${h_{a^i}}^\prime {f_{a^ib^j}}^{\prime} {h_{b^j}}^{\prime-1} = {h_{a^i}}^{\prime \prime} {f_{a^ib^j}}^{\prime\prime} {h_{b^j}}^{\prime\prime-1}$ on $W_{a^i} \cap V^\prime \cap V^{\prime\prime}$; so, we may define a 1-cochain $(f_{a^ib^j})$ of $GL_n\cO$, with respect to $\coprod_{i \in I}(W_{a^i})_{a^i \in I^{(i)}}$, by $f_{a^ib^j} = {h_{a^i}}^\prime {f_{a^ib^j}}^{\prime} {h_{b^j}}^{\prime-1}$ on $W_{a^i} \cap V^\prime$, and $f_{a^ib^j} =  {h_{a^i}}^{\prime \prime} {f_{a^ib^j}}^{\prime\prime} {h_{b^j}}^{\prime\prime-1}$ on $W_{a^i} \cap V^{\prime\prime}$. We can also see that the 1-cochain $(f_{a^ib^j})$ satisfies $f_{a^ib^j}f_{b^jc^k} = f_{a^ic^k} \alpha_{ijk}$, so it defines a holomorphic $\alpha$-twisted vector bundle on $V$; that is, an object $F$ in the category $i\mathrm{Vect}_\alpha^n(\cU|_{V}, \cO)$. The $\alpha$-twisted bundle defined by restricting $F$ to either $V'$ or $V''$ is isomorphic to the corresponding restriction of $E$ in the category $i\mathrm{Vect}_\alpha^n(\cU|_{V'}, \cC)$ or $i\mathrm{Vect}_\alpha^n(\cU|_{V''}, \cC)$, respectively. We note, however, that these $\alpha$-twisted vector bundles need not be isomorphic over $V$. 
    
    Now, consider the bundle $\Iso(E, F)$ of isomorphisms from $E$ to $F$ over $V$. As we know, this is a genuine, generalised principal bundle with structure group bundle $\Aut(E)$. The goal of this proof is achieved by finding a holomorphic $\Aut(E)$-principal bundle which is isomorphic to $\Iso(E, F)$. 
    
    Before we proceed to find such a bundle, let us substantiate the above claim. Firstly, and evidently, finding a holomorphic $\Aut(E)$-principal bundle that is topologically isomorphic to $\Iso(E,F)$ is equivalent to fitting a compatible holomorphic structure on the bundle $\Iso(E, F)$. Once we do that, the local holomorphic trivialisations of $\Iso(E, F)$ allow us to transport  holomorphic structure from $F$ to $E$, and we would be done. To be more precise, recall that on each open set $U_i$, $E$ and $F$ are both genuine $GL_n\CC$-principal bundles, so we only need to convince ourselves that $E|_{U_i}$ has a compatible holomorphic structure for any $i \in I$, if $\Iso(E, F)$ does; and indeed, we simply declare any trivialisation of the principal bundle $E|_{U_i}$ to be holomorphic whenever its image in $F|_{U_i}$, under any local holomorphic section of $\Iso(E, F)$, is a holomorphic trivialisation of $F$. 
    
    Let us now return to finding a holomorphic $\Aut(E)$-principal bundle that is isomorphic to $\Iso(E, F)$. We know that the $\Iso(E, F)$ has sections over $V'$ and $V''$, which means that it is isomorphic to $\Aut(E)$ when restricted to these subsets. So, $\Iso(E, F)$ may be fully defined by a 1-cocycle of $GL_n\cC$ with respect to the open cover $\{V', V''\}$ of $V$. As this open cover consists of only two open sets, such a 1-cocycle is nothing but a continuous section $\varphi: V' \cap V'' \to |\Aut(E)|$ where $|\Aut(E)|$ denotes the total space of the group bundle $\Aut(E)$. As $\Aut(E)$ is itself an $\Aut(E)$-principal bundle, and $V' \cap V''$ is holomorphically convex, Theorem \ref{Cartan's thm 1} applies to the section $\varphi$ -- it may be continuously deformed to a holomorphic section $\psi: V' \cap V'' \to |\Aut(E)|$. The section $\psi$ is also a 1-cochain -- and trivially, a 1-cocycle -- of $GL_n\cO$ with respect to the open cover $\{V', V''\}$ of $V$, and so, by Corollary \ref{cor: homotopic implies iso}, the holomorphic $\Aut(E)$-principal bundle it defines is isomorphic to $\Iso(E, F)$, as we sought.
    \end{proof}

    We can now prove the culminating result of this paper.

%\begin{centralthm}
   % \textit{Let $X$ be a reduced Stein space, and $\alpha \in H^2(X, \cO^*)$ be a torsion class. Then for every $n \geq 0$, the map $K_\alpha^{-n, \cO}(X) \to K_\alpha^{-n, \cC}(X)$ induced by inclusion is a natural isomorphism of groups.}
%\end{centralthm}
    
\noindent \textit{Proof of Main Theorem.}
Let us fix an open cover $\cU = (U_i)_{i \in I}$ of $X$ on which $\alpha$ can be represented as a completely normalised 2-cocycle. We show that the strict monoidal functor \[(i\mathrm{Vect}_\alpha(\cU, \cO), \square) \to (i\mathrm{Vect}_\alpha(\cU, \cC), \square)\] of simplicial symmetric monoidal categories, induced by the inclusion $\cO \to \cC$, is a weak equivalence of simplicial symmetric monoidal categories. By Theorem \ref{thm: K-theory equiv}, this will imply that the map $K_\alpha^{-n, \cO}(X) \to K_\alpha^{-n, \cC}(X)$ is a group isomorphism. To establish the result, we prove the following:
\begin{enumerate}
    \item[(1)] given $\alpha$-twisted holomorphic vector bundles $E, F \in \mathrm{ob(i\mathrm{Vect}_\alpha(\cU, \cO))}$, the induced map on hom-sets \[i\mathrm{Vect}_\alpha(\cU, \cO)(E, F) \to i\mathrm{Vect}_\alpha(\cU, \cC)(E, F)\] is a weak equivalence of simplical sets, and 
    \item[(2)] any $\alpha$-twisted topological vector bundle $E \in \mathrm{ob(i\mathrm{Vect}_\alpha(\cU, \cC))}$ is isomorphic to an $\alpha$-twisted holomorphic vector bundle $F \in \mathrm{ob(i\mathrm{Vect}_\alpha(\cU, \cO))}$. 
\end{enumerate}
Firstly, we can see that (1) is a direct consequence of the Theorem \ref{parametric Cartan}. Indeed, the map \[i\mathrm{Vect}_\alpha(\cU, \cO)(E, F) \to i\mathrm{Vect}_\alpha(\cU, \cC)(E, F)\] of hom-sets is, by definition,
\[ \mathrm{Sing}(\Gamma_\cO(\Iso(E, F))) \to \mathrm{Sing}(\Gamma_\cC(\Iso(E, F)))
    \] where $\Gamma_\cO(\Iso(E, F))$ (respectively, $\Gamma_\cC(\Iso(E, F))$) denotes the space of holomorphic (respectively, continuous) sections of the isomorphism bundle from $E$ to $F$ (see Definition \ref{def: cat twisted}). The isomorphism bundle between $\alpha$-twisted vector bundles is, per Definition \ref{def: cat twisted}, an untwisted generalised principal bundle,  with  $GL_n\CC$-torsors for fibres. As the singular functor defines a Quillen equivalence between the categories of topological spaces and simplicial sets, the map above is a weak equivalence of simplicial sets if and only if the map $\mathrm\Gamma_\cO(\Iso(E, F)) \to \Gamma_\cC(\Iso(E, F))$ is a weak equivalence of topological spaces. To this map we now apply the parametric Oka principle for generalised principal bundles (Theorem \ref{parametric Cartan}), which says that for a generalised principal bundle over a Stein space, the inclusion of the space of holomorphic sections into the space of continuous sections is a weak equivalence. This completes the proof of (1).

The proof of (2) is more algebraically involved; it follows similar arguments to the proof of \cite[Theorem B]{cartan1958espaces}, but for $\alpha$-twisted cocycles. Consider an $\alpha$-twisted topological vector bundle $E$. Say it is defined, per Definition \ref{def: cat twisted}, by the following data: an open cover $(W_{a^i})_{{a^i} \in I^{(i)}}$ of $U_i$ for each $i \in I$, and a 1-cochain $\left(g_{a^{i}b^{j}}\right)_{a^i\in I^{(i)}, b^j\in I^{(j)}, i,j \in I}$ of the sheaf $GL_n\cC$, with respect to the open cover $\coprod_{i \in I}(W_{a^i})_{a^i \in I^{(i)}}$ of $X$. (Again, whenever clear, we will simply write $(g_{a^ib^j})$, $(W_{a^i})$ and $\coprod_{i \in I}(W_{a^i})$.)

It is evident that every sufficiently small $V \subset X$ is good (take $V$ small enough that it is contractible).
%and $E$ and $\Aut(E)$ are isomorphic to the trivial bundle $V \times GL_n\CC$ over $V$. 
By ordinary principle bundle theory, on such an open set $V$, every $\Aut(E)$-principal bundle -- in particular, every holomorphic $\Aut(E)$-principal bundle -- is topologically trivial. Our goal now is to show that $X$ is good.

The proof proceeds in two parts using a bootstrap argument. In the first part, we establish that if $V$ is the interior of a special kind of compact set $K \subset X$, then $V$ is good. Then, in the second part, we take an exhaustion of $X$ by such compact sets $K$, and use the initial case to extend the property to all of $X$.

\noindent \textit{Part 1:} Let $K \subset X$ be a compact set for which there exists an isomorphism $\phi$ from a neighbourhood of $K$ to an analytic subspace $A$ of an open, relatively compact box $\Gamma \subset \CC^n$ (such an isomorphism exists by classical results; see \cite[\S 1]{cartan1958espaces}). Denote by $V \subset X$ the interior of $K$. We want to show that $V$ is good. For this, let us view $\Gamma$ as a subset of $\RR^{2n}$ and write it as a product $J_1 \times J_2 \times \cdots \times J_k$, where $J_m$ is an interval in $\RR$, and $k$ is the real dimension of $\Gamma \subset \RR^{2n}$. Cover each edge $J_m$ by open intervals $(J_m^\omega)_{\omega \in \Omega_m}$, where $\Omega_m$ is a finite set $\{1, 2, \ldots, q_m\}$, so that 
\begin{enumerate}
    \item the collection of sets $\Gamma_{\omega_1, \ldots, \omega_k}= J_1^{\omega_1} \times \cdots \times J_k^{\omega_k}$, for $(\omega_1, \ldots, \omega_k) \in \Omega_1 \times \cdots \times \Omega_k$, is a finite open cover of $\Gamma$, and
    \item each open set $V_{\omega_1, \ldots, \omega_k}= V \cap \phi^{-1}(\Gamma_{\omega_1, \ldots, \omega_k})$ is good. 
\end{enumerate}
Note that the collection of the sets $V_{\omega_1, \ldots, \omega_k}$ is an open cover of $V$.
For $0 \leq m \leq k$, and an  arbitrary fixed $(\omega_{m+1}, \ldots, \omega_k) \in \Omega_{m+1} \times \cdots \times \Omega_k$, consider \[V_{\omega_{m+1}, \ldots, \omega_k} : = \bigcup\limits_{(\omega_1, \ldots, \omega_m) \in \Omega_1 \times \cdots \times \Omega_m} V_{\omega_1, \ldots, \omega_k}.\] Then, by Lemma \ref{again compact box induction}, and induction on $m$, we find that each set $V_{\alpha_{m+1}, \ldots \alpha_k}$ is good -- therefore, $V$ is good.

\noindent \textit{Part 2:}  We fix an exhaustion $(K_n)_{n \geq 0}$ of $X$ consisting of special compact sets $K_n \subset X$ -- that is, a compact set which has a neighbourhood that is isomorphic to an analytic subspace of an open, relatively compact box $\Gamma \subset \CC^m$. It is always possible to find such an exhaustion of a Stein space $X$: any neighbourhood of the holomorphically convex hull $\hat{K}$ of a holomorphically convex, compact subset $K \subset X$, contains a special compact set which contains $\hat{K}$. Assume that each compact set $K_n$ lies in the interior $V_{n+1}$ of $K_{n+1}$. From Part 1, we know that $V_n$ is good for each $n \geq 0$. We now wish to prove that $X$ is good.

Assume, as we may, that each open set $W_{a^i}$ is relatively compact, and that if $W_{a^i} \cap V_n$ is nonempty, then $W_{a^i} \subset V_{n+1}$. We know that on each $V_n$, there exists a holomorphic $\alpha$-twisted bundle $F_n$, given by a 1-cochain $\left(f_{a^{i}b^{j}}^n\right)$ of the sheaf $GL_n\cO$, along with 0-cochain $(c_{a^{i}}^n)$ of $GL_n\cC$, both with respect to the open cover $\coprod_{i \in I}(W_{a^i}|_{V_n})$ of $V_n$, such that \[{c^n_{a^{i}}}\, g_{a^{i}b^{j}} {(c^n_{b^{j}})}^{-1}= f^n_{a^{i}b^{j}}.\]

It is evident that the $\alpha$-twisted holomorphic vector bundles $F_n$ and $F_{n+1}$, each being isomorphic to $E$ on $V_n$ and $V_{n+1}$ respectively, are themselves isomorphic in the category $i\mathrm{Vect}_\alpha(\cU|_{V_n}, \cC)$.  The cochain $((c_{a^i}^n)^{-1})$ is, of course, a continuous section of the generalised principal bundle $\Iso(F_n, E|_{V_n})$; similarly, $({c_{a^i}}^{n+1})$ is a continuous section of $\Iso(E|_{V_{n+1}}, F_{n+1})$. Composing $((c_{a^i}^n)^{-1})$ with $({c_{a^i}}^{n+1})$ gives us an isomorphism from $F_n$ to $F_{n+1}|_{V_n}$ -- expressly, it is given by the 0-cochain $(d_{a^i}^n)$, where $d_{a^i}^n = c_{a^i}^{n+1}(c_{a^i}^n)^{-1}$. Since $(d_{a^i}^n)$ is a section of the generalised principal bundle $\Iso(F_n, F_{n+1}|_{V_n})$, Theorem \ref{Cartan's thm 1} tells us that this section may be deformed, through continuous sections of $\Iso(F_n, F_{n+1}|_{V_n})$, to a holomorphic section. That is, there exists a family $(D_{a^i}^n(t))$, $t \in [0,1]$, of 0-cochains of $GL_n\cC$ with respect to $\coprod_{i \in I}(W_{a^i}|_{V_n})$, such that 
\begin{enumerate}
\item $f^n_{a^{i}b^{j}} = D_{a^i}^n(t)^{-1}f^{n+1}_{a^{i}b^{j}}D_{b^j}^n(t)$ for all $t \in [0,1]$,
\item $D_{a^i}^n(0) = d_{a^i}^n = c_{a^i}^{n+1}(c_{a^i}^n)^{-1}$, and
\item $(D_i^n(1))$ takes values in the sheaf $GL_n\cO$.
\end{enumerate}
This of course implies that the section of $\Iso(F_n, E|_{V_n})$ given by $((c_{a^i}^n)^{-1})$ may be continuously deformed, through such sections, to the section of $\Iso(F_n, E|_{V_n})$ defined by $((c_{a^i}^{n+1})^{-1}D_{a^i}^n(1))$. We will inductively modify the $\alpha$-twisted bundles $F_n$ so that they remain holomorphic, and isomorphic to $E$ over $V_n$, but ensuring further that $F_{n+1} \vert_{V_{n-2}} =F_{n} \vert_{V_{n-2}}$. 

Assume we have done this for $F_m$ for all $m \leq n$. Let $c_{a^i}^{n+1}$ remain the same if $W_{a^i} \cap V_{n-1}$ is empty; if not, replace $c_{a^i}^{n+1}(x)$ with $(D_{a^i}^n(\lambda(x))(x))^{-1}c_{a^i}^{n+1}(x)$ for $x \in W_{a^i} \cap V_n$, where $\lambda: K_n \to [0,1]$ is a continuous function with $\lambda = 0$ on $V_{n-2}$ and $\lambda = 1$ on $K_n \setminus V_{n-1}$.

By induction, we modify the 0-cochains $(c_{a^i}^n)$ for each $n \geq 0$ in this way. If we denote by $(c_{a^i})$ the pointwise limit, then the $\alpha$-twisted 1-cocycle  ${c_{a^{i}}}\, g_{a^{i}b^{j}} {c_{b^{j}}}^{-1}$ takes values in the sheaf $GL_n\cO$ -- in other words, it defines an $\alpha$-twisted holomorphic vector bundle $F$ which is isomorphic to $E$, as we sought. This concludes the proof of (2).

Lastly, say we are given another reduced Stein space $Y$, along with a holomorphic map $\phi: Y \to X$. For naturality, we want the diagram 
\[\begin{tikzcd}
    K_\alpha^{\cO}(\cU) \arrow{r}  \arrow[swap]{d}{K (\phi_\cO^*)}  & K_\alpha^{\cC}(\cU) \arrow{d}{K (\phi_\cC^*)} \\
    K_{\phi^*\alpha}^{\cO}(\phi^{-1}(\cU)) \arrow{r} & K_{\phi^*\alpha}^{\cC}(\phi^{-1}(\cU))
    \end{tikzcd}
    \] to commute in the category of topological spaces. Indeed, this follows from Propositions \ref{prop: Lambda}, \ref{prop: funct 2}, and \ref{prop: funct 3}; it is a direct consequence of the way the categories and morphisms above are defined. 
\qed

\section{Comparison with existing definitions of higher twisted K-theory} \label{sec 5}

In this final section, we outline an approach for comparing our definition of twisted K-theory for Stein spaces with existing notions of topological twisted K-theory in the literature, and sketch the key ideas and steps that would lead to proving consistency between these definitions. Closely related ideas, though in the context of Banach algebras and using different methods, were studied by M.~Paluch in his PhD thesis \cite{paluch1991algebraic}.

Let $X$ be an \textit{admissible} space -- that is, a paracompact, Hausdorff space which can be expressed as a countable union of compact subsets, and is of finite topological dimension. In particular, a reduced Stein space $X$ is admissible. Let $\alpha \in H^3(X,\ZZ)$ be torsion. Firstly, fix an $\alpha$-twisted bundle $E$, given with respect to an open cover $\cU$ of $X$ on which $\alpha$ can be represented\footnote[1]{It is not true in general that there exists an $\alpha$-twisted vector bundle over $X$ for a given $\alpha$; however, we interest ourselves only in spaces which admit $\alpha$-twisted vector bundles. For example, when $X$ has the homotopy type of CW complex, an $\alpha$-twisted vector bundle may always be found \cite{karoubi2012twisted}.}. Several established approaches to twisted K-theory (for example, \cite{atiyah2004twisted,karoubi2012twisted,bouwknegt2002twisted}) restrict attention to compact spaces or finite covers whenever the twisting class $\alpha \in H^3(X, \mathbb{Z})$ is torsion. It can be shown that when $X$ is compact, the definition we provide of topological twisted K-theory is consistent with these constructions; in other words, our definition is a natural generalisation to the class of admissible spaces, compact or otherwise. The particular aim of this section is to sketch a proof of this claim.

For an $\alpha$-twisted vector bundle $E$, let $\End(E)$ be the \textit{endomorphism bundle} of $E$ over $X$ -- this is an algebra bundle whose fibres are isomorphic $\End(\CC^n)$, which is the algebra $M_n\CC$ of $n \times n$ matrices with complex coefficients; it is constructed analogously to the group bundle $\Aut(E)$, whose fibres are isomorphic to $GL_n\CC$ (see \cite{karoubi2012twisted} for more details). There are two main steps to proving the claim above. The first step is to show that 
the following topologically enriched, symmetric monoidal categories are equivalent:
\begin{enumerate}
    \item $(i\mathrm{Vect}_\alpha(X, \cC), \oplus)$, the topological category of $\alpha$-twisted vector bundles on $X$, with only isomorphisms, per Definition \ref{def: cat twisted};
    \item $(i\mathrm{Mod}(X)_{\mathrm{End}(E)}, \oplus)$, the topological category of untwisted, complex vector bundles with right $\mathrm{End}(E)$-action, with direct sum, and only isomorphisms;
    \item $(i\mathrm{Mod}^{\mathrm{fg,proj}}(X)_{A}, \oplus)$, the topological category of finitely generated projective $A$-modules, with direct sum, and only isomorphisms, where $A$ is the Fr\'echet algebra $\Gamma(X, \mathrm{End}(E))$. 
\end{enumerate} 

  %\item $(i\mathrm{Mod}^{\mathrm{fg,proj}}(X)_{\cC,A}, \oplus)$, the topological category of finitely generated projective $\cC(X)$-modules with right $A$-action, with direct sum, and only isomorphisms, where $A$ is the Fr\'echet algebra $\Gamma(X, \mathrm{End}(E))$;

We know, by Theorem \ref{thm: K-theory equiv}, that if these topological symmetric monoidal categories -- equivalently viewed as simplical symmetric monoidal categories -- are even weakly equivalent, they induce isomorphic K-theories per Definition \ref{k theory def}. 

The second step is to show that the K-theory of $(i\mathrm{Mod}^{\mathrm{fg,proj}}(X)_{A}, \oplus)$ defined by the enriched $S^{-1}S$ construction (Definition \ref{k theory def}) agrees with any of the equivalent topological twisted K-theory definitions found in the literature (for example, the definition given in \cite{karoubi2012twisted}), when $X$ is a compact Hausdorff space. 

\subsection*{\normalsize{Step 1}}
The equivalence between $(i\mathrm{Vect}_\alpha(X, \cC), \oplus)$ and $(i\mathrm{Mod}(X)_{\mathrm{End}(E)}, \oplus)$ is detailed in \cite[Theorem 3.5]{karoubi2012twisted}. Note that while the construction of the category of $\alpha$-twisted vector bundles is slightly different in \cite{karoubi2012twisted}, the arguments in this case apply identically. Moreover, although only the discrete, non-symmetric-monoidal case is treated, it is evident that the topological structure of the hom-sets is preserved,  as is the symmetric monoidal structure, which is just direct sum for both categories. The equivalence between the categories $(i\mathrm{Mod}(X)_{\mathrm{End}(E)}, \oplus)$ and $(i\mathrm{Mod}^{\mathrm{fg,proj}}(X)_{A}, \oplus)$ follows from an adaptation of the Serre-Swan-Vaserstein theorem for admissible spaces \cite{vaserstein1986vector}. This theorem generalises the well-known Serre-Swan theorem from compact to admissible spaces; and, similar to its classical counterpart, a key step in the proof is showing that for any finite-dimensional complex vector bundle $F$, there exists a vector bundle epimorphism from a trivial bundle to $F$, which is then used to construct a finitely generated projective $\cC(X)$-module. Our context requires the following modification. For admissible spaces, finite-dimensional complex vector bundles are generated by finitely many global sections. This applies in particular to $\mathrm{End}(E)$, which implies that for any finite-dimensional complex vector bundle $F$ with right $\mathrm{End}(E)$-action, there exists a vector bundle epimorphism $\oplus_{m} \mathrm{End}(E) \to F$ for some $m \in \mathbb{N}$. From this point, the remainder of the Serre-Swan proof adapts naturally, with $A$-projectivity taking the role of $\cC(X)$-projectivity. Moreover, as with the previous equivalence of categories, it is not difficult to note that the topological structures of hom-sets are preserved, and the symmetric monoidal structure is respected.

\subsection*{\normalsize{Step 2}}
We begin by noting that the Fr\'echet algebra $A= \Gamma(X, \End(E))$ is unital and locally multiplicatively convex. For $n \geq 0$, let $K^{\mathrm{fr}}_n(A)$ be the Grothendieck group $K_0(A)$ of isomorphism classes of finitely generated projective $A$-modules when $n$ is even, and let $K^{\mathrm{fr}}_n(A)= \pi_0(GL^\mathrm{top}(A))$ when $n$ is odd. Equivalently, $$K^{\mathrm{fr}}_n(A) = \pi_n(K_0(A) \times BGL^\mathrm{top}(A))$$ for $n \geq 0$.  In \cite[Theorem 7.7]{phillips1991k}, N.~C.~Phillips shows that for any unital Fréchet algebra $A$, the above definition agrees with his more general definition of K-theory for Fr\'echet algebras presented in the same paper -- what appears to be the first systematic treatment of Fr\'echet algebra K-theory. Note that when $X$ is compact, $A = \Gamma(X, \End(E))$ is a unital $C^*$-algebra, and the above definition of $K_n^{\mathrm{fr}}(A)$ precisely coincides with established definitions of $\alpha$-twisted K-theory $K_\alpha^{-n}(X)$ of $X$ as found in, for example, \cite{atiyah2006twisted}. It remains to show that for all $n \geq 0$, \[K^{\mathrm{fr}}_n(A) = \pi_n\lvert N\,(i\mathrm{Mod}^{\mathrm{fg,proj}}(X)_{A})^{-1}i\mathrm{Mod}^{\mathrm{fg,proj}}(X)_{A}\rvert.\]

Let $M$ be the topological monoid given by taking the skeletal subcategory of the topological category $(i\mathrm{Mod}^{\mathrm{fg,proj}}(X)_{A}, \oplus)$; that is, \[M = \coprod_{[P]}\mathrm{Aut}^\mathrm{top}(P),\] as $P$ runs over the set of isomorphism classes of finitely generated projective $A$-modules.
In this case, $\Omega B(BM)$ is an infinite loop space, and the natural map $BM \to \Omega B(BM)$ is a group completion. It is well known that the output of an infinite loop space machine (such as the $S^{-1}S$ process) is unique up to homotopy, and topological symmetric monoidal categories are acceptable inputs to these machines \cite[pp.~331, 338]{weibel2013k}. This implies that, for all $n \geq 0$,
\[ \pi_n\lvert N\,(i\mathrm{Mod}^{\mathrm{fg,proj}}(X)_{A})^{-1}i\mathrm{Mod}^{\mathrm{fg,proj}}(X)_{A}\rvert = \pi_n (\Omega B(BM)).
\]

Firstly, for $n=0$, there is nothing for us to do -- we know that the $K_0$-output \[\pi_0\lvert N\,(i\mathrm{Mod}^{\mathrm{fg,proj}}(X)_{A})^{-1}i\mathrm{Mod}^{\mathrm{fg,proj}}(X)_{A}\rvert\] of the $S^{-1}S$ process equals the Grothendieck group $K_0(A)$ of isomorphism classes of finitely generated projective $A$-modules; indeed, the $\pi_0$ case effectively ignores the topological enrichment of $i\mathrm{Mod}^{\mathrm{fg,proj}}(X)_{A}$, and treats it as a discrete symmetric monoidal category. 

We argue as follows for $n \geq 1$. Since the monoidal functor \[f: i\mathrm{Mod}^{\mathrm{free}}(X)_A \to i\mathrm{Mod}^{\mathrm{fg,proj}}(X)_{A},\] where $i\mathrm{Mod}^{\mathrm{free}}(X)_A$ is the subcategory of free $A$-modules, is cofinal, for an arbitrary connected component $\Omega B (BM)_0$ of $\Omega B(BM)$,

\begin{align*}
  H_n(\Omega B (BM)_0) &= \coprod_{[P]} H_n(BM) \tag{$*$} \label{eq:groupcompletion}\\
  &= \coprod_{[P]} H_n(B\mathrm{{Aut}^{top}}(P)) \\
  &= \coprod_{[F]} H_n (B\mathrm{{Aut}^{top}}f(F)) \\
  &= \coprod_{[F]} H_n (B\mathrm{{Aut}^{top}}(F)) \\
  &= \coprod_{m \geq 0} H_n (BGL^{\mathrm{top}}_m(A)) \\
  &= H_n(BGL^{\mathrm{top}}A),
\end{align*}
where $F$ is a free $A$-module, and the subscripts $[F]$ and $[P]$ indicate colimits over the \textit{translation categories} (see \cite{weibel2013k}) of $\pi_0(B(i\mathrm{Mod}^{\mathrm{free}}(X)_A))$ and $\pi_0(B(i\mathrm{Mod}^{\mathrm{fg,proj}}(X)_{A}))$ respectively (this is a standard cofinality argument; see, e.g., \cite[Theorem 4.11]{weibel2013k}). The line ($*$) is the group completion theorem, as in \cite[Theorem 3.2.1]{adams1978infinite}. Since $\Omega B(BM)_0$ is a connected, group-like H-space, we deduce that $\Omega B (BM)_0$ and $BGL^{\mathrm{top}}A$ are homotopy equivalent spaces. Finally, as the path components $\Omega B(BM)$ are homotopy equivalent to one another, $\pi_n(\Omega B (BM)) = \pi_n(K_0(A) \times BGL^{\mathrm{top}}A) = K^{\mathrm{fr}}_n(A)$, as we sought.

\bibliography{bibliography}
\end{document}